\documentclass[leqno]{amsart}
\usepackage{amsmath,amsfonts,amsthm,amssymb,indentfirst,epic,url,centernot,graphics}

\newtheorem{theorem}{Theorem}
\newtheorem{lemma}[theorem]{Lemma}
\newtheorem{corollary}[theorem]{Corollary}
\newtheorem{proposition}[theorem]{Proposition}
\newtheorem{definition}[theorem]{Definition}
\newtheorem{question}[theorem]{Question}
\newtheorem{example}[theorem]{Example}

\newtheorem{convention}[theorem]{Convention}

\renewcommand{\r}{\mathrm}
\newcommand{\N}{\mathbb{N}}
\newcommand{\Z}{\mathbb{Z}}
\newcommand{\Q}{\mathbb{Q}}
\newcommand{\R}{\mathbb{R}}
\newcommand{\md}{_\mathrm{md}}
\newcommand{\gp}{_\mathrm{gp}}

\newcommand{\lang}{\begin{picture}(5,7)
\put(1.2,2.5){\rotatebox{45}{\line(1,0){6.0}}}
\put(1.2,2.5){\rotatebox{315}{\line(1,0){6.0}}}
\end{picture}\kern.16em}
\newcommand{\rang}{\kern.1em\begin{picture}(5,7)
\put(.1,2.5){\rotatebox{135}{\line(1,0){6.0}}}
\put(.1,2.5){\rotatebox{225}{\line(1,0){6.0}}}
\end{picture}}

\newcommand{\langlang}{\lang\kern-.3em\lang}
\newcommand{\rangrang}{\rang\kern-.24em\rang}

\newcommand{\smallcoprod}{\begin{picture}(10,6)
\put(1.8,.6){\line(1,0){6.6}}
\thicklines
\put(3.6,.6){\line(0,1){6}}
\put(6.6,.6){\line(0,1){6}}
\end{picture}}

\begin{document}

\title{On monoids, $\!2\!$-firs, and semifirs}
\thanks{After publication of this note, updates, errata,
related references etc., if found, will be recorded at
\url{http://math.berkeley.edu/~gbergman/papers}.
}


\author{George M. Bergman}
\address{University of California\\
Berkeley, CA 94720-3840, USA}
\email{gbergman@math.berkeley.edu}

\begin{abstract}
Several authors have studied the question of when the
monoid ring $DM$ of a monoid $M$ over a ring $D$ is a right and/or
left fir (free ideal ring), a semifir, or a $\!2\!$-fir (definitions
recalled in~\S\ref{S.Def}).
It is known that for $M$ nontrivial, a necessary condition for
any of these properties to hold is that $D$ be a division ring.
Under that assumption, necessary and sufficient conditions on $M$
are known for $DM$ to be a right or left fir, and various
conditions on $M$ have been proved necessary {\em or} sufficient for
$DM$ to be a $\!2\!$-fir or semifir.

A sufficient condition for $DM$ to be a semifir is that
$M$ be a direct limit of monoids which are free products
of free monoids and free groups.
Warren Dicks has conjectured that this is also necessary.
However  F.\ Ced\'{o} has given an example of a monoid $M$ which
is not such a direct limit, but
satisfies all the known necessary conditions for $DM$ to be a semifir.
It is an open question whether for this $M,$ the rings $DM$
are semifirs.

We note here some reformulations of the known necessary conditions
for a monoid ring $DM$ to be a $\!2\!$-fir or a semifir,
motivate Ced\'{o}'s construction and a variant
thereof, and recover Ced\'{o}'s results for both constructions.

Any homomorphism from a monoid $M$ into $\Z$ induces a
$\!\Z\!$-grading on $DM,$
and we show that for the two monoids just mentioned,
the rings $DM$ are ``homogeneous semifirs'' with respect to all such
nontrivial $\!\Z\!$-gradings; i.e., have (roughly) the property that
every finitely generated homogeneous one-sided ideal is free
of unique rank.

If $M$ is a monoid such that $DM$ is an $\!n\!$-fir,
and $N$ a ``well-behaved'' submonoid of $M,$ we
prove some properties of the ring $DN.$
Using these, we show that for $M$ a monoid such that $DM$
is a $\!2\!$-fir, mutual commutativity is an equivalence
relation on nonidentity elements of $M,$ and each equivalence
class, together with the identity element, is a directed union
of infinite cyclic groups or of infinite cyclic monoids.

Several open questions are noted.
\end{abstract}
\maketitle

\section{Definitions, and overview}\label{S.Def}

Rings are here associative and unital.

We recall that a ring $R$ is called a {\em free ideal ring},
or {\em fir}, if all left ideals and all right
ideals of $R$ are free as $\!R\!$-modules, and free
$\!R\!$-modules of distinct ranks are non-isomorphic.
The original motivating examples are the free associative algebras
$k\lang X\rang$ over fields $k.$
Many related constructions
are known to have the same property, for instance,
the completions $k\langlang X\rangrang$ of these algebras,
group algebras of free groups, and more generally,
monoid algebras of coproducts (a term I prefer to
``free product'') of free groups and free monoids, and the
generalizations of these examples with a division ring $D$ in place
of the field $k.$
The firs also include the classical principal
ideal domains, and various noncommutative generalizations of these,
e.g., Ore polynomial rings over division rings.
For much of what is known about these classes of rings, see \cite{FRR+}.

For any positive integer $n,$ a ring $R$ such that
every left ideal generated by $\leq n$ elements is free,
and free modules of ranks $\leq n$ have unique rank,
is called an {\em $\!n\!$-fir}\/; this
condition turns out to be left-right symmetric.
A ring which is an $\!n\!$-fir for all positive integers $n$ is
called a {\em semifir}.

For any $n,$ the class of $\!n\!$-firs is closed under
taking direct limits; hence so is the class of semifirs.
It follows that if $M$ is a direct limit of monoids
each of which is a coproduct of a free group and a free monoid, and
$D$ is a division ring, then $DM$ is a semifir.
These are the only examples of monoid rings of nontrivial
monoids that are known to be semifirs, and Dicks has
conjectured that they are the only examples
\cite[paragraph before Theorem~4.4]{WD+AS}.

Various conditions on a monoid $M$ have been shown necessary
for a monoid ring $DM$ to be a semifir.
In \S\ref{S.conditions} we recall these, and note reformulations
of some of them.
Ferran Ced\'{o} made the most recent addition to this
list of conditions in \cite{FC}.
He gives there examples of monoids satisfying
the previously known conditions but not his new ones,
thus showing that the previous conditions were not sufficient.
He then describes a monoid which satisfies all the known necessary
conditions, including his, but which is not a direct limit
of coproducts of free groups and free monoids.
It is not known whether this $M$ has monoid rings
$DM$ which are semifirs; if so, it would be a
counterexample to Dicks' conjecture.
In \S\ref{S.C_eg_props} below, we show that for this $M,$ the rings $DM$
do satisfy a graded version of the semifir condition.

\section{Conditions on a monoid $M$}\label{S.conditions}

Necessary and sufficient conditions
are known for a monoid ring to be a left or
right fir (\cite{RWW}, \cite{IK}, cf.\ \cite{FC+AP});
so it is only for the semifir condition, and the
weaker $\!n\!$-fir conditions, that such questions are open.
Moreover, $\!1\!$-firs are simply the rings without zero
divisors, a condition of a different flavor from the $\!n\!$-fir
conditions for higher $n,$ so this note will in general only consider
conditions for a monoid ring to be an $\!n\!$-fir when $n\geq 2.$
Finally, it is known that a monoid ring $RM$ where $M\neq\{1\}$
cannot be a $\!2\!$-fir unless $R$ is a division ring, so
(with a brief exception in \S\ref{S.cm})
we will not discuss monoid rings over non-division-rings.
We therefore make

\begin{convention}\label{Cv.D}
For the remainder of this note, $D$ will represent an
arbitrary division ring.
Thus, statements we make which refer to $D$ will be understood to be
asserted for all division rings $D.$

If $M$ is a monoid, then $DM$ will denote the monoid ring of $M$
with coefficients in $D,$ i.e., the ring of formal finite
linear combinations of elements of $M$ with coefficients
in $D,$ with the obvious addition, and with multiplication
defined so that elements of $D$ commute with elements of $M,$
while the \textup{(}generally noncommutative\textup{)}
internal multiplicative structures of $M$ and $D$
are retained.
\end{convention}

For $M$ a monoid, its {\em universal group} is the
group obtained by universally adjoining inverses to all elements of $M.$
A monoid need not embed in its universal group;
easy counterexamples are monoids with instances of non-cancellation,
$a\,b=a\,b'$ or $b\,c=b'\,c$ with $b\neq b'.$
(Right and left cancellation are, in fact, the simplest of an infinite
family of conditions, obtained by A.\,I.\,Mal'cev, which together are
necessary and sufficient for a monoid to be embeddable in a group.
See \cite{AIMSg1} and \cite{AIMSg2}, or the exposition in
\cite[\S VII.3]{PMC.UA}.)

It is not hard to show that a group $G$ is a direct limit of free
groups if and only if every finitely generated subgroup of $G$ is free;
such a $G$ is called {\em locally free}.
Dicks and Schofield \cite[Theorem~4.4, and the second sentence of the
proof of that theorem]{WD+AS} obtain the following strong results
for any monoid $M$ such that $DM$ is a semifir.
\begin{equation}\begin{minipage}[c]{35pc}\label{d.G_l_f}
The universal group of $M$ is locally free.
\end{minipage}\end{equation}
\begin{equation}\begin{minipage}[c]{35pc}\label{d.M<G}
$M$ embeds in its universal group.
\end{minipage}\end{equation}

The condition that a ring $R$ be an $\!n\!$-fir can be expressed
as saying that every $\!n\!$-term relation $\sum_1^n a_i\,b_i=0$
holding with $a_i,\,b_i\in R$ can be (in a sense that will be recalled
in \S\ref{S.semifir}) ``trivialized''.
In monoid rings $DM,$ the easy examples of such relations
are $\!2\!$-term relations, arising in
one way or another from equalities between two products in $M.$
Hence it is not surprising that most known necessary
conditions for $DM$ to be a semifir are consequences of the
$\!2\!$-fir condition.
Though~\eqref{d.G_l_f} and~\eqref{d.M<G} are not so obtained,
I don't know any cases where $DM$
is a $\!2\!$-fir but does not satisfy~\eqref{d.G_l_f} and~\eqref{d.M<G}.
Indeed, I do not know the answer to

\begin{question}\label{Q.2=>semi}
If $M$ is a monoid such that $DM$ is a $\!2\!$-fir,
must $DM$ be a semifir?
\end{question}

For commutative rings, the conditions of being a
$\!2\!$-fir and of being semifir are equivalent.
Indeed, given an ideal $I$ generated by $n>2$ elements $a_1,\dots,a_n,$
not all zero, in a commutative
$\!2\!$-fir, we note that the ideal generated by $a_1$ and $a_2$
must be free, hence since a commutative ring has no ideals free on
more than one generator, must be generated by a single element $a,$
so $I$ is generated by the $n-1$ elements $a,\,a_3,\dots,a_n;$ and
repeating this reduction, we find that it is free on one generator;
moreover, over commutative rings, all free modules have unique rank.
The same argument works over any right or left Ore ring;
so it is only among non-Ore rings that the distinctions among the
$\!n\!$-fir conditions for different values of $n>1$ arise.

It is shown in \cite{WD+PM} that if $G$ is a group such
that if $DG$ is an $\!n\!$-fir, then $G$ must have the
property that every $\!n\!$-generator subgroup is free.
If the converse were known to be true, then the fact that there are
groups having this property for $n=2$ but not for larger $n$ \cite{A+O},
\cite{B+S}, \cite[Examples on p.\,289]{WD+PM}
would give a negative answer to Question~\ref{Q.2=>semi}.
But though many examples are known of rings that
are $\!n\!$-firs but not $\!n{+}1\!$-firs
(e.g., \cite[Proposition~4.2]{PMC.depII},
\cite[Theorems~6.1 and 6.2]{cPu}), I am aware of none that are
group rings or monoid rings.

Nevertheless, let us hedge our bets, and supplement
Question~\ref{Q.2=>semi} with a weaker version, which
might have a positive answer if that question does not.

\begin{question}\label{Q.2=>semi_if}
If $M$ is a monoid satisfying~\eqref{d.G_l_f}
and~\eqref{d.M<G}, such that $DM$
is a $\!2\!$-fir, must $DM$ be a semifir?
\end{question}

We shall now recall some conditions that have been shown necessary for
$DM$ to be a $\!2\!$-fir.
When we give a condition in the form
of an implication, e.g., ``$(ab=ac)\implies(b=c)$'',
this will be understood to be quantified universally
over the elements of $M$ referred to.

Menal \cite{PM} proves that if $DM$ is a $\!2\!$-fir,
then $M$ satisfies the next three
conditions, of which the first two are together
expressed by saying $M$ is ``rigid''.
\begin{equation}\begin{minipage}[c]{35pc}\label{d.cancel}
$M$ is cancellative; that is,
$(ab=ac)\implies(b=c)\Longleftarrow(bd=cd).$
\end{minipage}\end{equation}
\begin{equation}\begin{minipage}[c]{35pc}\label{d.cap}
$aM\cap bM\neq\emptyset\implies aM\subseteq bM$ or $aM\supseteq bM.$
\end{minipage}\end{equation}
\begin{equation}\begin{minipage}[c]{35pc}\label{d.a=cad>}
If $a=cad,$ and $a$ is not invertible, then $c=d=1.$
\end{minipage}\end{equation}

Note that in the presence of~\eqref{d.cancel}, condition~\eqref{d.cap}
is left-right symmetric: it says that any relation
$ac=bd$ holding in $M$ can be refined either as $(bf)c=b(fc)$
(i.e., with $a$ having the form $bf,$ so that
by cancellation, $d=fc),$ or as $a(ed)=(ae)d,$ in $M.$

It is known that~\eqref{d.cancel} and~\eqref{d.cap} together
imply~\eqref{d.M<G}.
(Indeed, a somewhat stronger statement is proved in \cite{RD}
-- stronger because it assumes a weaker, but more complicated,
hypothesis than~\eqref{d.cap}.
A proof of the implication as just stated is given
in~\cite[Theorem~0.7.9]{FRR+}.)
Hence condition~\eqref{d.M<G}, stated above as necessary
for $DM$ to be a semifir, is in fact necessary for it
to be a $\!2\!$-fir (and hence can
be dropped from Question~\ref{Q.2=>semi_if}).

To the necessary conditions~\eqref{d.cancel}-\eqref{d.a=cad>}
for $DM$ to be a $\!2\!$-fir,
Ced\'{o} \cite[Proposition~2]{FC} adds the next two.
\begin{equation}\begin{minipage}[c]{35pc}\label{d.ab=cad>}
If $ab=cad,$ then either $c=1,$ or there exists $b'$ such that
$ab'=c^n$ for some nonnegative integer~$n.$
\end{minipage}\end{equation}
\begin{equation}\tag{\ref{d.ab=cad>}$'$}\begin{minipage}[c]{35pc}\label{d.ba=cad>}
If $ba=cad,$ then either $d=1,$ or there exists $b''$ such that
$b''a=d^n$ for some nonnegative integer $n.$
\end{minipage}\end{equation}

These conditions have some known reformulations, which we will now
develop.

Let us begin by looking at the case of~\eqref{d.ab=cad>} where $d=1.$
It is not hard to get from this (we shall do so in the proof of
Lemma~\ref{L.ab=ca} below) a more detailed statement:
\begin{equation}\begin{minipage}[c]{35pc}\label{d.ab=ca>}
If $ab=ca,$ and $b,$ $c$ are not both $1,$ then there exist
$e,f\in M$ and a nonnegative integer $n$ such that $a=(ef)^n e,$
$b=fe,$ and $c=ef.$
\end{minipage}\end{equation}

For some insight into this condition, let us see what we
can deduce from a relation $ab=ca$ using only~\eqref{d.cancel}
and~\eqref{d.cap}.
By~\eqref{d.cap}, one of $a$ and $c$ is a left divisor of the other.
If $a$ is a left divisor of $c,$ we can write $c=af,$
whence $b=fa,$ and on setting $e=a,$ we find that we have
the $n=0$ case of the conclusion of~\eqref{d.ab=ca>}.
This leaves the case where $c$ is a left divisor of $a;$ say $a=ca'.$
In that case, substituting into the given equation and
cancelling $c$ on the left, we get $a'b=ca',$ which, of
course, again leads to two possibilities.
In one, $c=a'f,$ and on taking $e=a',$ we see that $c=ef,$
so $a=ca'=efe,$ and we have the $n=1$ case of the conclusion
of~\eqref{d.ab=ca>}.
In the other case, writing $a'=ca'',$ we get a new equation $a''b=ca''.$
And so on.
So what condition~\eqref{d.ab=ca>} says is that this process
terminates after finitely many steps; in other words, that $a$
is not left divisible by arbitrarily large powers of $c$ --
except (if one examines the argument) when
$a,$ $b$ and $c$ are all invertible, in which case one has, at
every step, a choice whether to terminate the process or continue it.
Let us exclude this case by requiring $a$ to be noninvertible.
Thus, the idea of~\eqref{d.ab=ca>} seems to be the condition:
\begin{equation}\tag{\ref{d.ab=ca>}$'$}\begin{minipage}[c]{35pc}\label{d.ab=ca>'}
If $ab=ca,$ where $c\neq 1$ and $a$ is not invertible, then $a$ is not
left divisible in $M$ by all positive integer powers of $c.$
\end{minipage}\end{equation}

This has, of course, a dual:
\begin{equation}\tag{\ref{d.ab=ca>}$''$}\begin{minipage}[c]{35pc}\label{d.ab=ca>''}
If $ab=ca,$ where $b\neq 1$ and $a$ is not invertible, then $a$ is not
right divisible in $M$ by all positive integer powers of $b.$
\end{minipage}\end{equation}

The reader might find it helpful, at this point,
to jot down on a piece of paper the statements of
the numbered conditions
given so far,~\eqref{d.G_l_f}-\eqref{d.ab=ca>''},
for easy reference as he or she reads further;
I will refer to them frequently.
(Subsequent numbered displays, in contrast, will mostly
be referenced only near where they appear.)

The next result establishes the relations among the last few conditions
discussed.

\begin{lemma}[{\cite[Lemma~2]{RA+FC}}]\label{L.ab=ca}
If $M$ is a monoid satisfying~\eqref{d.cancel}
and\textup{~\eqref{d.cap}}, then\\[.4em]
\hspace*{17em}$\eqref{d.ab=ca>'}\iff
\eqref{d.ab=ca>}\iff
\eqref{d.ab=ca>''},$\\
and\\
\hspace*{15.8em}$\eqref{d.ab=cad>}\iff
(\eqref{d.a=cad>}\wedge\eqref{d.ab=ca>})\iff
\eqref{d.ba=cad>}.$
\end{lemma}

\begin{proof}
To prove that \eqref{d.ab=ca>'}$\Longleftarrow$\eqref{d.ab=ca>},
suppose~\eqref{d.ab=ca>} holds, and
that $ab=ca,$ where $c\neq 1$ and $a$ is not invertible.
By~\eqref{d.ab=ca>}, $a$ will left divide $c^{n+1}$ for some $n.$
Now if $a$ were left divisible by all positive powers
of $c,$ it would in particular be left divisible by $c^{n+2},$
so $c^{n+1}$ would be left divisible by $c^{n+2},$ so
(since $M$ satisfies~\eqref{d.cancel})
$c$ would be invertible; hence so would $a,$ as a left
divisor of $c^{n+1},$ contradicting our assumption.

For the reverse implication, again suppose $ab=ca,$
now assuming $b$ and $c$ not both $1.$
If $a$ is invertible, we can get the
conclusion of~\eqref{d.ab=ca>} by taking $n=0,$
$e=a,$ and $f=ba^{-1}=a^{-1}c.$
If not, then~\eqref{d.ab=ca>'} implies that
there is a largest nonnegative $n$ such that $c^n$ left divides
$a,$ so we can write $a=c^n e$ where $e\in M$
is {\em not} left divisible by $c.$
Substituting into our given relation, and cancelling $c^n$
from the left, we get $e\,b=c\,e,$ and since $e$ is not left
divisible by $c,$~\eqref{d.cap} says $c$ is left divisible by $e;$
say $c=ef.$
Cancellation now gives $b=fe,$ and substituting back, we
get $a=c^n e=(ef)^n e,$ giving the conclusion of~\eqref{d.ab=ca>}.

The implications $\eqref{d.ab=ca>}\iff\eqref{d.ab=ca>''}$
follow from the above by left-right symmetry.

To prepare for the proof of the second line of equivalences, let us
examine some consequences of~\eqref{d.ab=cad>} for a general relation
\begin{equation}\begin{minipage}[c]{35pc}\label{d.ab=cad}
$a\,b\ =\ c\,a\,d$
\end{minipage}\end{equation}
with $c\neq 1.$
Condition~\eqref{d.ab=cad>}
says that $a$ left divides some nonnegative power of $c;$
let $c^n$ be the least {\em positive} power of $c$ which $a$
left divides.
Thus, $c^n$ will be a common right multiple of $a$ and $c^{n-1},$
and $a$ will not be a proper left divisor of $c^{n-1}$ (by
our minimality assumption on $n$ if $n\geq 2,$ or
by the fact that $c^0=1$ has no proper left divisors if $n=1).$
Hence by~\eqref{d.cap}, $a$ must be a right multiple of $c^{n-1};$ say
$a=c^{n-1}e.$
On the other hand, as noted,
$c^n$ is a right multiple of $a,$ say $c^n=af.$
This gives $c^n=c^{n-1}ef;$ so by~\eqref{d.cancel},
\begin{equation}\begin{minipage}[c]{35pc}\label{d.c=}
$c\ =\ e\,f,$
\end{minipage}\end{equation}
and substituting into $a=c^{n-1}e,$
\begin{equation}\begin{minipage}[c]{35pc}\label{d.a=e...e}
$a\ =\ (ef)^{n-1}e.$
\end{minipage}\end{equation}
Substituting~\eqref{d.c=}
and~\eqref{d.a=e...e} into~\eqref{d.ab=cad}, and cancelling
$(ef)^{n-1} e$ on the left, we get
\begin{equation}\begin{minipage}[c]{35pc}\label{d.b=fed}
$b\ =\ f\,e\,d.$
\end{minipage}\end{equation}

To get \eqref{d.ab=cad>}$\implies$\eqref{d.a=cad>},
we now consider the case of~\eqref{d.ab=cad} where $b=1.$
If $c=1$ (the case excluded in the above discussion),
then~\eqref{d.ab=cad} becomes $a=ad,$ which by~\eqref{d.cancel}
gives $d=1,$ confirming~\eqref{d.a=cad>}.
If $c\neq 1,$ then we have~\eqref{d.b=fed}, which for $b=1$
implies (in view of~\eqref{d.cancel}), that $f,$ $e$
and $d$ are invertible, hence by~\eqref{d.a=e...e}, that $a$
is invertible, the case about which~\eqref{d.a=cad>} makes no assertion,
completing the proof that $M$ satisfies~\eqref{d.a=cad>}.

To see that \eqref{d.ab=cad>}$\implies$\eqref{d.ab=ca>},
we note that in the case $d=1$ of~\eqref{d.ab=cad}, the
consequences~\eqref{d.c=}-\eqref{d.b=fed}
are precisely the conclusion of~\eqref{d.ab=ca>}.

To show, conversely, that
$\eqref{d.a=cad>}\wedge\eqref{d.ab=ca>}\implies\eqref{d.ab=cad>},$
note that given a relation~\eqref{d.ab=cad}, the elements
$b$ and $d$ have a common left multiple, hence
by~\eqref{d.cap} one is a left multiple of the other.
If $b=b'd,$ then we can cancel $d$ from
the two sides of~\eqref{d.ab=cad}, getting a relation
to which we can apply~\eqref{d.ab=ca>} to get the desired conclusion.
If $d=d'b,$ we can similarly cancel a $b$ and apply~\eqref{d.a=cad>},
and conclude either $c=1,$ or $a$ is invertible.
The former is one of the alternative conclusions
of~\eqref{d.ab=cad>}, while the latter is the $n=0$ case of
the other alternative.

Since~\eqref{d.a=cad>} and~\eqref{d.ab=ca>}
are both left-right symmetric, the equivalence of their
conjunction with~\eqref{d.ab=cad>}
also implies the equivalence of that conjunction with the dual
statement,~\eqref{d.ba=cad>}.
\end{proof}

We remark that in monoids $M$ for which
$DM$ is a semifir, one {\em can} have arbitrarily large
powers of one noninvertible element left dividing another.
For instance, in the direct limit of the maps of free monoids
$\lang x,\,y_0\rang\to\lang x,\,y_1\rang\to\dots\to
\lang x,\,y_n\rang\to\dots,$ where each map sends $y_i$ to
$x\,y_{i+1}\,x,$ the element $y_0$ is infinitely divisible
by $x$ on both sides.
But~\eqref{d.ab=ca>'} and~\eqref{d.ab=ca>''} tell us,
somewhat mysteriously, that such infinite divisibility is
excluded for elements appearing in certain slots of a
relation~$a\,b=c\,a.$

Going back to~\eqref{d.a=cad>}, let us record a
generalization of that condition, which we will use in \S\ref{S.cm}.

\begin{lemma}\label{L.a1a2>}
Let $M$ be a monoid satisfying~\eqref{d.cancel}-\eqref{d.a=cad>}.
Then for every natural number $n,$ $M$ also satisfies
\begin{equation}\begin{minipage}[c]{35pc}\label{d.a1a2>}
If $a_1\dots a_n=c_0\,a_1\,c_1\,a_2\dots\,c_{n-1}\,a_n\,c_n,$
where $a_1,\dots,a_n$ are all noninvertible,\\
then $c_0=\dots=c_n=1.$
\end{minipage}\end{equation}
\end{lemma}

\begin{proof}
The $n=0$ case (with the hypothesis understood
to be $1=c_0)$ is a tautology, and
the $n=1$ case is~\eqref{d.a=cad>},
so let $n>1,$ assume inductively that
the desired result is known for $n-1,$ and suppose
we are given $a_1,\,a_2,\,\dots\,,\,a_n,$
and $c_0,\,c_1,\,\dots\,,\,c_n,$ satisfying the conditions stated.
We note that if $c_0=1,$ we can cancel $a_1$ on the left, and our
inductive assumption gives the desired conditions
on the remaining $c_i;$ and the analogous argument works if $c_n=1.$

Now by~\eqref{d.cap} applied to the relation
$a_1\cdot(a_2\dots a_n) = (c_0 a_1 c_1)\cdot(a_2\dots\,c_{n-1}a_n c_n),$
one of $a_1$ and $c_0 a_1 c_1$ must left divide the other;
so we can either write $a_1 d = c_0 a_1 c_1$ or $a_1 = c_0 a_1 c_1 d.$
In the latter case, we can apply~\eqref{d.a=cad>} to this relation,
and, in particular, get $c_0=1,$ which we have noted
gives our desired inductive step.
In the former case, we substitute $a_1 d$ for the $c_0 a_1 c_1$ in
our given relation and cancel $a_1$ from the two sides, getting
a relation to which our inductive assumption applies.
In particular, this gives $c_n=1,$ which we have noted also
gives the desired inductive step.
\end{proof}

\section{Moving toward Ced\'{o}'s example}\label{S.C_eg_motiv}

In this section we shall motivate Ced\'{o}'s example of a
monoid $M$ which satisfies the necessary
conditions~\eqref{d.G_l_f}-\eqref{d.ab=ca>}
for $DM$ to be a semifir, but is not a direct limit of
coproducts of free monoids and free groups.
In \S\ref{S.C_eg} we describe it formally, along with a
slight variant construction; in \S\ref{S.C_eg_props} we prove
that both monoids satisfy~\eqref{d.G_l_f}-\eqref{d.ab=ca>},
and in \S\ref{S.semifir}, obtain
some approximations to the statement that the rings $DM$ are semifirs.

Part of our development must, of course, be a proof that
our monoids $M$ are {\em not} direct limits of
coproducts of free monoids and free groups; so
let us begin by obtaining a necessary condition for a monoid
to be such a direct limit.
The problem of constructing a monoid that fails to
satisfy that condition but does
satisfy \eqref{d.G_l_f}-\eqref{d.ab=ca>}
will then motivate the examples.

We begin with a result about a wider class of coproduct monoids.

\begin{lemma}\label{L.abd=bae>}
Let $N$ and $N'$ be cancellative monoids,
such that $N'$ has no invertible elements other than $1,$
and suppose that in the coproduct monoid $N\smallcoprod N',$ we have
elements satisfying
\begin{equation}\begin{minipage}[c]{35pc}\label{d.abd=bae}
$a\,b\,d\ =\ b\,a\,e,$
\end{minipage}\end{equation}
where $d$ and $e$ lie in $N,$ while
$a$ and $b$ do not both lie in $N.$
Then $d=e.$
\end{lemma}

\begin{proof}
Let elements of $N\smallcoprod N'$ be written as
alternating strings of nonidentity elements of $N$ and of $N'.$
When we speak of the {\em initial} (respectively, {\em final})
$\!N\!$-factor of an element, we will mean the initial (final)
term of that string if this is an element of $N,$ or $1$ if it is not.

Let us begin by handling the case where one of $a,$ $b$ lies in $N.$
By symmetry, we can assume this is $a,$ so by hypothesis, $b\notin N.$
Letting $r$ be the initial $\!N\!$-factor
of $b,$ and comparing the initial $\!N\!$-factors of
the two sides of~\eqref{d.abd=bae}, we get $ar = r,$
so as $N$ is cancellative, $a = 1,$ so another
application of cancellativity gives $d=e$ as claimed.

Now assume neither $a$ nor $b$ lies in $N.$
Since $N'$ has no invertible elements, multiplication of two
elements of $N\smallcoprod N'$ cannot cause $\!N'\!$-factors
to cancel, and thus allow $\!N\!$-factors that
these had separated to interact; hence from~\eqref{d.abd=bae}
we can conclude that $a$ and $b$ have the same initial $\!N\!$-factor.
Call that $v,$ and
call their final $\!N\!$-factors $w$ and $w'$ respectively, writing
$a=v\,a_0\,w,$ $b = v\,b_0\,w',$ where $a_0$ and $b_0$ both have initial
and final $\!N\!$-factors~$1.$
Note that if $w=w',$ then by considering the final $\!N\!$-factors
in~\eqref{d.abd=bae}, and cancelling, we get the desired
conclusion; so we will complete the proof by assuming
$w\neq w',$ and getting a contradiction.

Substituting our expressions for $a$ and
$b$ into~\eqref{d.abd=bae}, and cancelling the
common initial factor $v,$ we get
\begin{equation}\begin{minipage}[c]{35pc}\label{d.a0...e}
$a_0\,w\,v\,b_0\,w' d\ =\ b_0\,w' v\,a_0\,w\,e.$
\end{minipage}\end{equation}

By the above equality, the family (set with multiplicity) of internal
(i.e., neither initial nor final) $\!N\!$-factors occurring on
the two sides of~\eqref{d.a0...e} must be the same.
Now by assumption, $wv$ and $w'v$ cannot both be $1.$
If one of them is $1,$ then we get different numbers
of internal $\!N\!$-factors on the two sides
of~\eqref{d.a0...e}, so that case is excluded.
If neither is $1,$ then the family of internal $\!N\!$-factors
on the left-hand side is the union (with multiplicity)
of the families of internal $\!N\!$-factors of $a_0$ and of $b_0,$
together with the single additional factor $wv,$ while on the right
we have the same with $w'v$ in place of $wv.$
Hence $wv=w'v,$ contradicting our assumption that $w\neq w',$
and completing the proof.
\end{proof}

\begin{corollary}\label{C.ab=bag>}
Let $M$ be a direct limit of monoids each of which is a coproduct
of a free group and a free monoid \textup{(}or more generally,
each of which is a coproduct of an arbitrary group, and a cancellative
monoid which has no invertible elements other than $1,$ and which
admits a homomorphism $h$ into an abelian group, such
that $h$ carries no nonidentity element to $1).$

Then in $M,$
\begin{equation}\begin{minipage}[c]{35pc}\label{d.ab=bag>}
If $a\,b = b\,a\,g,$ where $a$ and $b$ not both invertible, then $g=1.$
\end{minipage}\end{equation}
\end{corollary}

\begin{proof}
To see that the parenthetical generalized hypothesis does indeed cover
the original hypothesis, note that on a free monoid, the
degree function in the free generators has the properties
required of $h.$

Note further that it suffices to prove~\eqref{d.ab=bag>}
when $M$ is itself a free product of a group and a monoid
of the indicated sort.
For if $a\,b=b\,a\,g$ in a direct limit of such monoids, one can lift
$a,$ $b$ and $g$ to one of the monoids whose limit is being
taken, find a later step in the limit process where the
indicated equality holds, apply~\eqref{d.ab=bag>} in
that monoid, and conclude that the given case
of~\eqref{d.ab=bag>} holds in~$M.$

So let us assume that $M$ is the coproduct of a
group $N,$ and a cancellative monoid $N'$ of the indicated sort.
Note that if we apply to both sides of our given relation
the homomorphism $f:M\to N'$ that kills $N$ and acts as
the identity on $N',$ then apply the homomorphism $h$ of
our hypothesis, and then cancel $h(f(a))\,h(f(b))=h(f(b))\,h(f(a)),$ we
get $1=h(f(g)).$
Hence $1=f(g),$ hence no factor from $N'$
in the expression for $g$ can fail to equal $1,$ hence $g\in N.$
Hence we can apply Lemma~\ref{L.abd=bae>} with $d=1,$ $e=g$
(noting that the statement that $a,$ $b$ are not both invertible means
that they don't both belong to $N),$ and the result follows.
\end{proof}

Let us now try to find a monoid
$M$ that satisfies~\eqref{d.G_l_f}-\eqref{d.ab=ca>},
but not~\eqref{d.ab=bag>}.
A difficulty is that condition~\eqref{d.ab=ca>}, which
we want to be satisfied, and~\eqref{d.ab=bag>}, which we want
to fail, are closely related.
Given elements of a monoid
$M$ satisfying $a\,b=b\,a\,g$ with $a$ and $b$ not both
invertible and $g\neq 1,$ we can apply~\eqref{d.ab=ca>} with
the roles of $a,$ $b$ and $c$ played
by $b,$ $ag$ and $a,$ and (naming the $e$ and $f$ of the
conclusion of this instance of~\eqref{d.ab=ca>} $b'$ and
$a'),$ conclude that there
exists a nonnegative integer $n$ and elements $a',\,b'\in M$ such that
\begin{equation}\begin{minipage}[c]{35pc}\label{d.b=}
$b=(b'\,a')^n\,b',$\qquad $a=b'\,a',$\qquad $a\,g=a'\,b'.$
\end{minipage}\end{equation}

Substituting the second of these equations into
the last one, and interchanging the two sides of the result,
we get $a'b'=b'a'g,$ an equation of the sort we started with.
We can now make a second application of~\eqref{d.ab=ca>},
and conclude similarly that for some $n'\geq 0$ and $b'',c''\in M,$
we have $b'=(b''a'')^{n'} b'',$ $a'=b''a'',$ and $a'g=a''b''.$
And so on.

Once we see this
pattern, it is not hard to construct examples of this behavior.
Note that to get instances of $a\,b=b\,a\,g$ in a {\em group},
we need merely choose $a$ and $b;$ solving the equation
then gives $g=a^{-1}b^{-1}a\,b.$
So we can start in the free group on generators $x$ and $y,$
and let $M_0$ be its submonoid
$\lang x,\,y,\,x^{-1}y^{-1}x\,y\rang\md$
(where $\lang\dots\rang\md$ denotes generation as a monoid).
Here $x,$ $y,$ and $x^{-1}y^{-1}x\,y$
will play the roles of $a,$ $b$ and $g$ in the preceding discussion.
To insure that this relation $a\,b=b\,(a\,g)$
satisfies~\eqref{d.ab=ca>} in the monoid we are aiming
for, we should, by the above discussion,
map the above monoid into the free group on generators $x'$ and $y',$
by the map which, for some choice of $n,$ sends $x$ and $y$
respectively to $y'\,x'$ and $(y'\,x')^n y',$ and adjoin
$x'$ and $y'$ to the image of this monoid $M_0$ in that group.
Not surprisingly (given our preceding calculations where
the element $g$ kept its role at successive stages), we find
that this homomorphism maps $x^{-1}y^{-1}x\,y$
to $x'^{-1}y'^{-1}x'\,y'.$
Moreover, since $x$ and $y$ have been expressed in terms
of $x'$ and $y',$ we can now work in the monoid
$\lang x',\,y',\,x'^{-1}y'^{-1}x'\,y'\rang\md.$
We can repeat this process indefinitely, using possibly
different exponents $n,$ $n',$ etc.\ at successive stages, and let
$M$ be the direct limit of this process.
(This is not a direct limit of the form referred
to in Dicks' conjecture, since our successive submonoids of free
groups are not coproducts of a free group and a free monoid.)

Actually, there is no need to use different copies of
the free group on two generators at successive
stages; in other words, we can consider the group
homomorphisms in the above construction to be a chain
of endomorphisms $G\to G\to\,\dots\ ,$ where $G$ is the free
group on $x$ and $y.$
We find that these endomorphisms are actually automorphisms
of $G,$ since one can express $x$ and $y$ in terms of
$y\,x$ and $(y\,x)^n y.$
(The generator $y$ can be obtained
by left-multiplying $(y\,x)^n y$ by the $\!(-n)\!$-th
power of $y\,x,$ and then $x$ can be recovered from $y$ and $yx.)$
So the direct limit of the groups $G$ under these
maps will still be $G;$ and the direct limit of the
monoids $M_0$ will be the union of an increasing
chain of isomorphic copies of $M_0$ within~$G.$

We can get such a direct limit monoid $M$ using any sequence
of choices of $n,$ $n',$ $n'',$ etc..
What sequence shall we try?

The simplest choice is $n=n'=\dots=0,$ so that each of
our maps carries $x$ to $y\,x$ and $y$ to $y,$
But this is too simple.
Although the relation $x\cdot y=y\cdot xz$ behaves well in
the limit monoid, if
we re-partition it as $x\cdot y=yx\cdot z,$
we find that it does not satisfy~\eqref{d.cap}.

But going to the next simplest choice, $n=n'=\dots=1,$ we shall
see that the resulting monoid has the properties we want.
It is, in fact, the example developed by Ced\'{o}
in \cite[pp.\,128-131]{FC}.
We shall study that example, and a variant, in the next three sections.

\section{Ced\'{o}'s example, and a variant}\label{S.C_eg}
Below, we shall write $\lang\dots\rang\md$ for generation
or presentation as a monoid, $\lang\dots\rang\gp$ for generation
or presentation as a group.

Let us name the group we will be working in, and the submonoid thereof
from which we will build Ced\'{o}'s example as a direct limit.
\begin{equation}\begin{minipage}[c]{35pc}\label{d.G>M_0}
$G$ will denote the free group on two generators $x$ and $y.$\\
$M_0$ will denote the submonoid of $G$ generated by $x,\,y,$
and $z=x^{-1}y^{-1}x\,y.$
\end{minipage}\end{equation}

We now gather some information about $M_0.$

\begin{lemma}\label{L.M_0}
A presentation for the monoid $M_0$ of~\eqref{d.G>M_0}
is $\lang x,\,y,\,z\mid y\,x\,z=x\,y\rang\md.$

A normal form for its elements is given by all strings in the
symbols $x,$ $y$ and $z$ containing no substring $y\,x\,z.$

The universal group of $M_0$ is $G,$ with the inclusion of $M_0$ in $G$
as in~\eqref{d.G>M_0} as the universal map.
\end{lemma}

\begin{proof}
Clearly the relation $y\,x\,z=x\,y$ holds in $M_0.$
Since replacing substrings $y\,x\,z$ by $x\,y$ reduces the length
of any string, we can, by successive applications of that reduction,
transform any string into one containing no substring $y\,x\,z.$
We shall prove next that if two strings containing no such
substrings represent the same element of $G,$ then they are equal
as strings.

This will imply the presentation and
normal-form assertions of the lemma.
To get the final assertion, note that since
$M_0=\lang x,\,y,\,z\mid y\,x\,z=x\,y\rang\md,$ its universal
group will be $\lang x,\,y,\,z\mid y\,x\,z=x\,y\rang\gp.$
In that presentation,
we can solve the relation $y\,x\,z=x\,y$ for $z,$ getting
$z=x^{-1}y^{-1}x\,y,$ and so eliminate the generator
$z$ and that relation from the presentation.
Hence the universal group of $M_0$ is presented by the generators
$x,$ $y$ and no relations, and so is, indeed, $G.$

Turning to what we must prove, suppose $r=s$ were an
equality holding between distinct expressions in $x,$ $y$
and $z$ containing no substrings $y\,x\,z,$ chosen to minimize, among
such examples, the sum of the lengths of $r$ and $s.$
Clearly, the leftmost letters of $r$ and of $s$
cannot be the same.
When we map both sides of this relation into $G,$
the two sides must give the same reduced group words in $x$ and $y.$
Since no $z$ in $r$ or $s$ is preceded by the sequence $y\,x,$
the only simplifications that can occur, initially, are of the form
$x\,z=x\cdot(x^{-1}y^{-1}x\,y)\mapsto y^{-1}x\,y.$
We can then have further simplifications, in that the $y^{-1}$
with which one such product begins can cancel the $y$ with
which a preceding $z$ or the result of simplifying
a preceding $x\,z$ ends.
But we can handle this from the start, by applying the simplification
$(x\,z)^m\mapsto y^{-1}x^m\,y$ to all maximal strings $(x\,z)^m,$
and then, if such a string was preceded by a $z,$ making the additional
simplification $z(y^{-1}x^m\,y)\mapsto x^{-1}y^{-1}x^{m+1}y.$
We see that when we have done this,
the images of $r$ and $s$ will be reduced group words, and
hence must coincide.

In this process, we note that the $y^{-1}$ resulting from
the leftmost $\!z\!$'s in $r$ and in $s$ (if any) will not be affected
by any reduction, and so will appear in the final reduced word.
It follows that if either of $r$ or $s$ contained no $z,$
the same would have to be true of the other, and they would equal.
So they must both contain $z.$
Next, if their common reduced image in $G$ does not
begin with $x^{-1}$ or $y^{-1},$ but with some other letter
$u$ (namely, $x$ or $y),$ then the expressions $r$ and $s$ must both
begin with that letter $u,$ contradicting our observation
that (by minimality) they cannot begin with the same letter.
Finally, if their images both begin with $x^{-1},$ then in $M_0,$
both $r$ and $s$ begin with $z,$ while if both images begin with
$y^{-1}$ then they both begin with $x\,z,$ contradicting
that same observation.
This completes the proof of the lemma.
\end{proof}

Let me now introduce the variant construction I have alluded to.
The direct limit process (to be described below)
by which we get the final monoid will be the
same, but instead of starting with
$M_0=\lang x,\,y,\,z\rang\md\subseteq G$ as in~\eqref{d.G>M_0},
we will use the larger submonoid
\begin{equation}\begin{minipage}[c]{35pc}\label{d.M_1}
$M_1\ =\ \lang x,\,y,\,z,\,z^{-1}\rang\md\subseteq G,$
where again, $z=x^{-1}y^{-1}x\,y.$
\end{minipage}\end{equation}

The idea is that by making $z,$ i.e., the $g$ in our counterexample
to~\eqref{d.ab=bag>}, invertible, and hence ``a little more like $1$'',
we might keep $M$ from violating some as-yet-undiscovered requirement
for $DM$ to be a semifir.
It is not hard to adapt the proof of Lemma~\ref{L.M_0}, and get
the analog of that result for this monoid:

\begin{lemma}\label{L.M_1}
A presentation for $M_1$ is
$\lang x,\,y,\,z,\,z^{-1}\mid
y\,x\,z=x\,y,\ x\,y\,z^{-1}=y\,x,\ z\,z^{-1}=1=z^{-1}z\rang\md.$

A normal form for its elements is given by all strings in the
symbols $x,$ $y,$ $z$ and $z^{-1}$ containing no substrings
$y\,x\,z,$ $x\,y\,z^{-1},$ $z\,z^{-1}$ or $z^{-1}\,z.$

The universal group of $M_1$ is $G,$
with the inclusion of $M_1$ in $G$ as in~\eqref{d.M_1}
as the universal map.\qed
\end{lemma}

Essentially all our results about these examples will be proved for both
the limit monoid based on $M_0$ and the limit monoid based on $M_1.$
Here is the description of the map over which we will take our
direct limits, as motivated in the last four paragraphs of the
preceding section.
\begin{equation}\begin{minipage}[c]{35pc}\label{d.sigma}
$\sigma$ will denote the automorphism of $G$ which
acts by\quad $\sigma(x)=y\,x,\quad \sigma(y)=y\,x\,y,$
and which (it is easy to check) fixes $z$ and hence $z^{-1},$
and thus carries $M_0$ and $M_1$ into themselves.
\end{minipage}\end{equation}

We now name our direct limit monoids.
Because $\sigma$ is an automorphism of $G,$ we can take
the desired direct limits within $G.$
We let
\begin{equation}\begin{minipage}[c]{35pc}\label{d.M()}
\hspace*{-.2em}%
$M_{(0)}\ =$ the set of $a\in G$ such that
$\sigma^n(a)\in M_0$ for some (hence, for all sufficient large) $n.$\\
$M_{(1)}\ =$ the set of $a\in G$ such that $\sigma^n(a)\in M_1$
for some (hence, for all sufficient large) $n.$
\end{minipage}\end{equation}

Ced\'{o} introduces what we are calling $M_{(0)}$ in the
example beginning at \cite[p.\,128, after proof of Lemma~4.4]{FC}.
(He writes $\varphi$ for $\sigma^{-1},$ and describes the monoid
as the union in $G$ of the chain of submonoids $\varphi^n(M_0).$
His generators $r$ and $t$ are our $x$ and $y.)$

Note that
\begin{equation}\begin{minipage}[c]{35pc}\label{d.aut}
For $i=0,1,$ the restriction of $\sigma$ to the direct limit
monoid $M_{(i)}$ is a monoid automorphism.
\end{minipage}\end{equation}

\section{Basic properties of the above monoids}\label{S.C_eg_props}

From the final assertions of Lemmas~\ref{L.M_0} and~\ref{L.M_1},
we have

\begin{lemma}\label{L.M()<G}
$M_{(0)}$ and $M_{(1)}$ satisfy~\eqref{d.G_l_f} and~\eqref{d.M<G}
\textup{(}each having $G$ as its universal group\textup{)},
and hence~\eqref{d.cancel}.\qed
\end{lemma}

More work is required to get

\begin{lemma}\label{L.M().cap}
$M_{(0)}$ and $M_{(1)}$ satisfy~\eqref{d.cap}.
\end{lemma}

\begin{proof}
In verifying~\eqref{d.cap} for a given relation
\begin{equation}\begin{minipage}[c]{35pc}\label{d.ab=cd}
$a\,b\,=\,c\,d$
\end{minipage}\end{equation}
in $M_{(i)},$ we may clearly first apply an arbitrarily high power
of $\sigma$ to the elements of $M_{(i)}$ in question.
By doing so we can, to start with, bring such
elements into the submonoid $M_i,$ where we can write
them in the normal form of Lemma~\ref{L.M_0} or~\ref{L.M_1}.
Further applications of $\sigma$ may not preserve that normal
form, however: If an expression contains a sequence $x\,z$
or $y\,z^{-1},$ we see that on applying $\sigma,$ we get an
expression which can be reduced, thus removing an occurrence
of $z$ or $z^{-1}.$
On the other hand, application of $\sigma$ never introduces
new occurrences of $z$ or $z^{-1},$ so under successive
applications of $\sigma$ to an element, the
number of such factors eventually stabilizes.
From this fact, and the explicit formula for
$\sigma,$ one sees that after sufficiently many
applications of that map, every element of $M_{(i)}$
becomes and subsequently remains an element of $M_i$ in whose
normal-form expression
\begin{equation}\begin{minipage}[c]{35pc}\label{d.enough_sigma}
No $z$ is immediately preceded by an $x,$
no $z^{-1}$ is immediately preceded by a $y,$
every $x$ is immediately preceded by a $y,$
and if $y$ is the last letter in the
expression, that occurrence of $y$ is immediately preceded by an $x.$
\end{minipage}\end{equation}

So suppose we have a relation~\eqref{d.ab=cd} in $M_{(i)},$
such that $a,\,b,\,c,\,d$ lie in $M_{i},$
and their normal forms satisfy~\eqref{d.enough_sigma}.

If these normal-form expressions,
when multiplied together as on the two sides of~\eqref{d.ab=cd}, still
give normal-form expressions, then the normal-form expression for one of
$a,$ $c$ must be an initial substring of the other, and we have
the conclusion of~\eqref{d.cap}, as desired.
Moreover, in view of~\eqref{d.enough_sigma},
the only way these product-expressions
can fail to be in normal form is if the expression for $b$ or $d$
begins with a $z$ or $z^{-1},$ and the expression for $a,$ respectively
$c,$ ends with a symbol
or pair of symbols that interacts with this letter.

It is now easy to dispose of the case $i=1$
(i.e., the case where our monoid is $M_{(1)}).$
Since $z$ and $z^{-1}$ are invertible,
and multiplying on the right by an invertible element
does not affect right divisibility relations, we can take whatever
power of $z$ occurs at the beginning of $b$ or $d,$ and
move it to the end of $a,$ respectively $c,$
and then reduce the resulting expressions, and, if necessary,
apply $\sigma$ until~\eqref{d.enough_sigma} again holds.
Then by the preceding paragraph, one of these elements
will left divide the other in $M_0$ and hence
in $M_{(0)},$ establishing~\eqref{d.cap}.

So suppose $i=0.$
Since a reduction must occur in at least one of the
products $a\cdot b$ and $c\cdot d$
of~\eqref{d.ab=cd}, we can assume without loss of generality that
\begin{equation}\begin{minipage}[c]{35pc}\label{d.yxz}
The product $a\cdot b$ is reducible; i.e., we can write
$a=a'y\,x,$ $b=z\,b',$ where $a',$ $b'$ are in normal
form for $M_0.$
Moreover, if the product $c\cdot d$ is also reducible, and we similarly
write $c=c'y\,x,$ $d=z\,d',$ then we may assume the total degree
of $a$ in $x$ and $y$ is at least that of $c.$
\end{minipage}\end{equation}

Note that if we simplify $a\cdot b$ by the
reduction $a'y\,x\cdot z\,b'\mapsto a'x\,y\,b',$
then the resulting expression is in normal form.
For the only way it might not be is if the $y$ shown were immediately
followed by $x\,z$ in $b';$ but that would contradict the
assumption that $b$ satisfies~\eqref{d.enough_sigma}, in
particular, the condition
that every $x$ is immediately preceded by a $y.$
Hence
\begin{equation}\begin{minipage}[c]{35pc}\label{d.red_or}
In the notation of~\eqref{d.yxz}, the normal form of $a\,b$
is $a'x\,y\,b',$ and if the expression $c\cdot d$ is reducible,
its normal form is $c'x\,y\,d',$ where $\r{deg}(c')\leq\r{deg}(a').$
\end{minipage}\end{equation}

We now argue by cases.
If $c\cdot d$ is already in normal form, then $c$ must be an
initial substring of $a'x\,y\,b'.$
It will left divide $a$ if as a substring it is contained in $a',$
while it will be a right multiple of $a$ if it
contains $a'x\,y=a'y\,x\,z=a\,z;$ this leaves only the case $c=a'x.$
In this case we apply the automorphism $\sigma$ once more, and see that
$\sigma(c)=\sigma(a'x)=\sigma(a')y\,x,$ while
$\sigma(a)=\sigma(a'y\,x)=\sigma(a')y\,x\,y\,x\,y,$
so $\sigma(c)$ left divides $\sigma(a)$ in $M_0,$
hence $c$ left divides $a$ in $M_{(0)},$ as required.

On the other hand, if $c\cdot d$ is not initially in normal
form, then equating the normal forms of the two sides
of~\eqref{d.ab=cd}, we get $a'x\,y\,b'=c'x\,y\,d',$ with $c'$ having at
most the degree of $a'.$
If $c$ is not to left-divide $a,$ then $c'$ cannot equal $a',$
and we see that $c'x\,y$ must be an initial substring of $a'.$
Hence $a,$ being a right multiple of $a',$ is a right multiple
of $c'x\,y=c'y\,x\,z=c\,z,$ hence is a right multiple of $c.$
So again, the conclusion of~\eqref{d.cap} is satisfied.
\end{proof}

Much easier is

\begin{lemma}\label{L.M().a=cad>}
$M_{(0)}$ and $M_{(1)}$ satisfy~\eqref{d.a=cad>}.
\end{lemma}

\begin{proof}
Let $a=c\,a\,d$ in $M_{(i)},$ with $a$ non-invertible.
Again, we may assume without loss of generality that $a,\,c,\,d\in M_i.$

Now the total degree function in $x$ and $y,$ i.e., the
homomorphism from $G$ to the additive group of integers that
takes $x$ and $y$ to $1,$ has positive value on the generators
$x$ and $y$ of $M_i,$ but value $0$ on $z$ and $z^{-1}.$
Hence all elements of $M_i$ have positive degree, except those
in the submonoid (if $i=0)$ or subgroup (if $i=1)$ generated by $z.$
Hence applying the degree function to the given relation,
we see that both $c$ and $d$ must be (natural number or integer)
powers of $z;$ say
\begin{equation}\begin{minipage}[c]{35pc}\label{d.z^p,q}
$c=z^p,$ $d=z^q.$
\end{minipage}\end{equation}

Let us begin with the case where $a$ is also a power of $z.$
If $i=0,$ this says that our given equation
has the form $z^m=z^{p+m+q}$ with all of
$m,$ $p,$ $q$ nonnegative, so $p=q=0,$ so $c=d=1,$ as desired.
On the other hand, if $i=1,$ this case cannot occur,
since it would make $a$ invertible, contrary to our assumption.

Assuming $a$ is not a power of $z,$ its normal form in $M_i$ will
involve at least one occurrence of $x$ or $y.$
Now since the rules for reduction to normal form,
other than $z\,z^{-1}=1=z^{-1}z,$ only
affect $z$ and $z^{-1}$ when they occur to the {\em right} of
another letter, when we reduce $c\,a\,d$ to normal
form in $M_i$ the power (possibly $0)$ of $z$ occurring to the
left of the leftmost $x$ or $y$ will be exactly $p$ more than
the corresponding value in the normal form of $a.$
Since $a=c\,a\,d,$ this means $p=0.$
So $c=1,$ so $a=a\,d,$ so $d=1,$ as required.
\end{proof}

Finally, we want to prove that $M_{(0)}$ and $M_{(1)}$
satisfy~\eqref{d.ab=ca>}, equivalently,~\eqref{d.ab=ca>'}
or~\eqref{d.ab=ca>''}.
As with~\eqref{d.cap}, this will require delicate
considerations, because the condition
we want to prove is similar to the condition~\eqref{d.ab=bag>}
which we have arranged will not hold.
We will prove~\eqref{d.ab=ca>'}, i.e., that if $ab=ca,$
then $a$ is not left divisible
in $M_{(i)}$ by all positive integer powers of $c.$
In doing so, it is not sufficient to work with the images of our
elements in $M_{(i)}$ under a fixed power of $\sigma,$ since
it could happen that as we apply larger and
larger powers of $\sigma,$ the image of $a$ becomes divisible
by larger and larger powers of the image of $c.$
(Or differently stated, that larger and larger
submonoids $\sigma^{-n}(M_{(i)})$ contain more and more of the
elements $c^{-m}a.)$

The trick that will help us will be suggested by the following
easily verified observation,
though that observation will not be called on in the proof.

\begin{lemma}\label{L.sigma^1/2}
The automorphism $\sigma$ of $G$ is the square of an
automorphism $\sigma^{1/2},$ which acts by
\begin{equation}\begin{minipage}[c]{35pc}\label{d.sigma^1/2}
$\sigma^{1/2}(x)=y,\qquad \sigma^{1/2}(y)=y\,x.$
\end{minipage}\end{equation}
This automorphism sends $z$ to $z^{-1},$ and vice versa.\qed
\end{lemma}

Now writing down the first few terms of
the orbit of $x$ under $\sigma^{1/2},$
\begin{equation}\begin{minipage}[c]{35pc}\label{d.orbit}
$x,\quad y,\quad yx,\quad yxy,\quad yxyyx,\quad
yxyyxyxy,\quad yxyyxyxyyxyyx,\quad \dots$
\end{minipage}\end{equation}
we find that each term is the product of the two that precede.
Precisely, we have
\begin{equation}\begin{minipage}[c]{35pc}\label{d.sigma_Fib}
$(\sigma^{1/2})^{n+2}(x)\ =
\ (\sigma^{1/2})^{n+1}(x)\ (\sigma^{1/2})^{n}(x).$
\end{minipage}\end{equation}
This can be verified by noting that it holds
for $n=0,$ and applying $(\sigma^{1/2})^n$ to the resulting equation.
From this we see that the length of the term $(\sigma^{1/2})^n(x)$
is the Fibonacci number $F_{n+1}.$
Now the Fibonacci sequence is a linear combination of the sequences
of powers of $(1+\sqrt 5)/2$ and $(1-\sqrt 5)/2,$ and this line
of thought leads us to the following observation.

\begin{lemma}\label{L.delta}
Let $\tau=(1+\sqrt 5)/2,$ and let $\delta$ be the homomorphism from $G$
to the additive group of real numbers defined by
\begin{equation}\begin{minipage}[c]{35pc}\label{d.delta}
$\delta(x)=1,\qquad \delta(y)=\tau$\qquad
\textup{(}and thus $\delta(z)=0).$
\end{minipage}\end{equation}

Then for all $a\in G,$
\begin{equation}\begin{minipage}[c]{35pc}\label{d.delta_sigma}
$\delta(\sigma^{1/2}(a))=\tau\,\delta(a),$\quad and hence\quad
$\delta(\sigma(a))=\tau^2\,\delta(a).$
\end{minipage}\end{equation}

Hence $\delta$ assumes nonnegative values on all elements
of $M_{(i)},$ and positive values on all of these except powers of $z.$
\end{lemma}

\begin{proof}
The relation~\eqref{d.delta_sigma} holds because it holds on
the generators $x$ and $y$ of $G.$
The final assertion is seen by applying~\eqref{d.delta_sigma}
to the generators of $M_i,$ and then
to $M_{(i)}=\bigcup \sigma^{-n}(M_i).$
\end{proof}

We can now deduce

\begin{lemma}\label{L.ab=ca>}
The only elements $c\in M_{(i)}$ such that there exist elements
left or right divisible by all powers of $c$ \textup{(}i.e.,
such that $\bigcap c^n M_{(i)}$ or $\bigcap M_{(i)} c^n$
is nonempty\textup{)} are the invertible elements.
\end{lemma}

\begin{proof}
It is immediate from the last assertion of
Lemma~\ref{L.delta} that the only elements
$c$ which can possibly have either of the above properties are
the powers of $z.$
In $M_{(1)},$ these are invertible, so if $i=1,$ we are through.

In $M_{(0)},$ it remains to show
that no element is left or right divisible
by arbitrarily large powers of $z.$
This is equivalent to saying that for fixed $a\in M_0,$ the powers
of $z$ left or right dividing elements $\sigma^n(a)$ in $M_0$
are bounded as $n$ grows.

Now as noted in the proof of Lemma~\ref{L.M().a=cad>},
the string of $\!z\!$'s at the left-hand end of a product
of the generators of $M_0$ is not affected by reduction to normal form;
and neither is the length of that string changed by applying $\sigma;$
so that length gives an upper bound on the power of
$z$ left dividing $a,$ as required.

This is all we will need in what follows, but for
completeness, we have asserted the corresponding result for right
divisibility as well, so here is a sketch of the argument in that case.

The deviation from the case we have discussed
arises from the ability of some elements, when
multiplied on the right by one or more $\!z\!$'s, to ``absorb''
them, so that the product has a normal form not showing these right
factors.
An element whose normal form is $a\,y\,x$ can ``absorb'' one
$z$ in this way, since $a\,y\,x\,z=a\,x\,y,$ but since the
result ends in $y,$ it can absorb no more; and it is easy to
see that its images under all powers of $\sigma$ still end
in $y,$ and so can still absorb no $\!z\!$'s.
It follows that an element of $M_0$ whose normal form ends
with $z^m$ and no higher power of $z$
can be right divisible by no higher power
of $z$ than $z^{m+1},$ not only in $M_0$ but in $M_{(0)},$ completing
the proof.
\end{proof}

\begin{corollary}\label{C.M().ab=ca>}
$M_{(0)}$ and $M_{(1)}$ satisfy~\eqref{d.ab=ca>}, \eqref{d.ab=ca>'}
and~\eqref{d.ab=ca>''}.
\end{corollary}

\begin{proof}
By Lemma~\ref{L.ab=ca} it suffices to prove condition~\eqref{d.ab=ca>'}.
The left-divisibility statement of the above lemma proves this
for relations $a\,b=c\,a$ in which
$c$ is not invertible, while if $c$ is invertible but $\neq 1,$
such a relation can be written $a=c^{-1}\,a\,b,$ and is
then excluded by~\eqref{d.a=cad>}, proved in Lemma~\ref{L.M().a=cad>}.
\end{proof}

On the other hand,
the relation $y\,x\,z=x\,y$ gives us, as planned,

\begin{lemma}\label{L.not_bag}
$M_{(0)}$ and $M_{(1)}$ do not satisfy~\eqref{d.ab=bag>}.
Hence neither of them is a direct limit of monoids which are
coproducts of a free monoid and a free group.\qed
\end{lemma}

\section{Semifirs, and graded semifirs}\label{S.semifir}

We defined in \S\ref{S.Def} the class of rings called $\!n\!$-firs.
Let us recall here a useful reformulation of the $\!n\!$-fir condition.

In any ring $R,$ by an {\em $\!m\!$-term linear
relation} we will mean a relation $\sum_{i=1}^m u_i v_i = 0$
with $u_i,\,v_i\in R.$
We may write such a relation as $(u_i)\cdot (v_i)=0,$ where
$(u_i)$ and $(v_i)$ are the length-$\!m\!$ row vector and
height-$\!m\!$ column vector formed from the indicated elements.
We shall say our relation is ``trivial'' if for
each $i\in\{1,\dots,m\},$ either $u_i=0$ or $v_i=0.$
Given an $\!m\!$-term linear relation $(u_i)\cdot (v_i)=0$
and an invertible $m\times m$ matrix $T,$ let
$(u'_i)=(u_i)\,T,$ $(v'_i)=T^{-1}(v_i),$ so that the
given relation is equivalent to $(u'_i)\cdot (v'_i)=0.$
If this new relation is trivial, then
we shall say that $T$ {\em trivializes} the relation
$(u_i)\cdot (v_i)=0.$
It is not hard to show that a ring $R$ is an $\!n\!$-fir if and
only if every $\!m\!$-term linear relation with $m\leq n$
can be trivialized by some invertible $m\times m$ matrix
\cite[Theorem~2.3.1]{FRR+}.

Let us now generalize these ideas to rings graded by a group $A.$
Because the grading groups I have in mind are $\Z$ and $\Z\times\Z,$
I will write the operation of $A$ as addition, though in this
general discussion, we do not need $A$ to be commutative.
If $R$ is an $\!A\!$-graded ring, let us understand a
{\em homogeneous} linear relation in $R$ to be a relation
$\sum_{i=1}^m u_i v_i = 0$ such that for
some $\alpha_1,\dots,\alpha_m$ and $\beta$ in $A,$
we have $u_i\in R_{\alpha_i}$ and $v_i\in R_{-\alpha_i+\beta}$
for each $i;$
and let us say such a relation is {\em homogeneously trivializable}
if it is trivializable by an invertible matrix $T=((T_{ij}))$ with
$T_{ij}\in R_{-\alpha_i+\alpha_j}$ for each $i$ and $j.$
I shall call the graded ring $R$ a {\em homogeneous $\!n\!$-fir} if
every homogeneous linear relation of $\leq n$ terms is homogeneously
trivializable.
(Though the name might suggest a condition stronger
than being an $\!n\!$-fir, it is, roughly speaking, weaker.
It would be truly weaker if we merely required that
every homogeneous linear relation be trivializable, rather
than homogeneously trivializable.
I don't know whether homogeneous trivializability is a stronger
condition on a homogeneous relation than trivializability.)

If $h:M\to A$ is a homomorphism from a monoid to a group,
this induces an $\!A\!$-grading on the monoid ring $DM.$
Namely, for each $\alpha\in A,$ we let the homogeneous component of $DM$
of degree $\alpha,$ which we will write $DM_\alpha,$
be the span over $D$ of $\{c\in M\mid h(c)=\alpha\}.$

Now let $G$ be the free group on $x$
and $y,$ and $M_{(0)}$ and $M_{(1)}$ the
submonoids of $G$ so named in the preceding three sections.
The map $\delta: G\to\Z\times\Z$ of the next proof is, up
to isomorphism, the $\delta$
of Lemma~\ref{L.delta}, since the image $\Z+\tau\Z$ of
the latter map is isomorphic to $\Z\times\Z;$ but since
the embedding of that group in the reals, and the resulting
ordering, are irrelevant here, we use the more
abstract description.
Our first step toward our desired result on $\!\Z\!$-graded monoid
rings will be the following $\!\Z\times\Z\!$-graded result.

\begin{lemma}\label{L.M>ZxZ}
Let $M=M_{(0)}$ or $M_{(1)},$ and
let $\delta:G\to\Z\times\Z$ be the homomorphism taking
$x$ to $(1,0)$ and $y$ to $(0,1).$
Then for any division ring $D,$ the $\!\Z\times\Z\!$-graded
ring $DM$ is a $\!\delta\!$-homogeneous semifir.
\end{lemma}

\begin{proof}
Assuming the contrary, let
\begin{equation}\begin{minipage}[c]{35pc}\label{d.sum_uv}
$\sum_{i=1}^n u_i\,v_i\ =\ 0$
\end{minipage}\end{equation}
be an $\!n\!$-term homogeneous linear relation
which cannot be homogeneously trivialized; say with
$u_i\in DM_{\alpha_i}$ and $v_i\in DM_{-\alpha_i+\beta}$
$(\alpha_i,\beta\in\Z\otimes\Z).$
(From this point on, however, let us, for brevity, abbreviate
``trivialize homogeneously'' to ``trivialize'', and
understand that any invertible matrices by which we act on
our row and column vectors have all entries homogeneous,
of the appropriate degrees.)

Assume~\eqref{d.sum_uv} chosen to minimize $n.$
Then I claim that
\begin{equation}\begin{minipage}[c]{35pc}\label{d.no_0}
No invertible matrix $T$ can make either any component of $(u_i)\,T$
or any component of $T^{-1}(v_i)$ equal to zero.
\end{minipage}\end{equation}
For if $T$ did so, then $(u_i)\,T\cdot T^{-1}(v_i)=0$ would
become, on deleting the trivial
term, a linear relation with $n-1$ terms, which
by our minimality assumption could be trivialized,
yielding a trivialization of our original relation.

Now let us choose from among the finitely many elements of $M$
comprising the supports of $u_1,\dots,u_n,$ one element $a$ which is
not a proper right multiple of any of the others, i.e.,
which maximizes $a\,M.$
Say $a$ occurs in the support of $u_1.$
Let us write each $u_i$ as $a\,u'_i + u''_i,$
where the terms $u''_i$ have support consisting of elements
that are not right multiples of $a.$
By choice of $a,$ none of the elements of the
supports of the $u''_i$ are proper left divisors of $a$ either,
so as $M$ satisfies~\eqref{d.cap}, no right multiple of
any of these support-elements coincides with any right multiple
of $a;$ so projecting~\eqref{d.sum_uv} onto $a(DM)$ we get
$\sum a\,u'_i\,v_i=0,$ and, cancelling $a,$ we have $\sum u'_i\,v_i=0.$

Note that in this last relation, at least the summand $u'_1$
is nonzero; hence if this relation could be trivialized,
the trivializing transformation would leave at least one of the
left-hand factors nonzero, and hence make one of the right-hand
factors zero.
But this would contradict~\eqref{d.no_0};
so our relation $\sum u'_i\,v_i=0$ is still non-trivializable.
Renaming the $u'_i$ as $u_i,$ and recalling that the homogeneous
degree of $u_i$ is denoted $\alpha_i,$ we have reduced
ourselves to the situation where in~\eqref{d.sum_uv},
\begin{equation}\begin{minipage}[c]{35pc}\label{d.alpha_1}
$\alpha_1=(0,0),$ and $u_1$ has $1\in M$ in its support.
\end{minipage}\end{equation}

So far, everything we have said would be valid for any monoid $M$
satisfying~\eqref{d.cancel} and~\eqref{d.cap} and
given with a homomorphism $\delta$ into a group $A.$
Let us now note that for the two monoids $M_{(0)}$ and $M_{(1)}$
that we are interested in, the submonoids
\begin{equation}\begin{minipage}[c]{35pc}\label{d.N=}
$N\ =\ \delta^{-1}(\{(0,0)\})$
\end{minipage}\end{equation}
are, respectively, the cyclic submonoid generated by $z,$ and
the cyclic subgroup generated by $z;$ so that the rings
$DN$ are the (not necessarily commutative) right and
left principal ideal domains $D[z]$ and $D[z,z^{-1}].$
I claim that using this observation and~\eqref{d.no_0}, we can
complement~\eqref{d.alpha_1} with the statement
\begin{equation}\begin{minipage}[c]{35pc}\label{d.only_alpha_1}
$\alpha_i\neq (0,0)$ for all $i>1.$
\end{minipage}\end{equation}
Indeed, if this were not true, say we had $\alpha_2=(0,0).$
In the right principal ideal domain $DN,$ the elements $u_1$
and $u_2$ must be right linearly dependent, and such a linear
dependence relation in a principal right ideal domain
(which is, in particular, a right fir, hence a $\!2\!$-fir)
can be trivialized by an invertible $2\times 2$ matrix,
which will turn $(u_1,\,u_2)$ into a row with one term zero.
Now extending this to an
$n\times n$ matrix by attaching an $n{-}2\times n{-}2$ identity matrix,
we would get a contradiction to~\eqref{d.no_0}.
This proves~\eqref{d.only_alpha_1}.

Let us now examine $u_1.$
If $M=M_{(0)},$ then $u_1\in D[z]$ is a polynomial
$p(z)$ having nonzero constant term by~\eqref{d.alpha_1}.
If $M=M_{(1)},$ then $u_1\in D[z,\,z^{-1}]$ is a Laurent polynomial;
but by applying a matrix $T$ which agrees with the identity
except in its $(1,1)$ entry, for which we
use an appropriate power of the
invertible element $z,$ we may reduce
to the case where $u_1$ is again an ordinary polynomial $p(z)\in D[z]$
with nonzero constant term.
Thus, we may assume that this is the case
whether $M$ is $M_{(0)}$ or $M_{(1)}.$
In either case, let $d$ be the degree of $p(z).$

We shall now apply to~\eqref{d.sum_uv}
an invertible matrix $T$ whose effect on
$(u_i)$ is to subtract from each $u_i$ with $i>1$ a certain
right $\!DM\!$-multiple of $u_1=p(z).$
I claim that we can do this so that the supports of the
resulting terms consist of elements $a\in M$ in which the power
of $z$ occurring at the far left in $a$ lies between $0$ and $d-1.$
(To make precise what we mean by the power of $z$ occurring
at the far left, recall that for all
sufficiently large $r,$ $\sigma^r(a)$ will lie in $M_0$ or $M_1,$
where we have a normal form for elements, so we can speak of the
power of $z$ that occurs at the left end in that normal form; and
as noted in the last paragraph of the proof of Lemma~\ref{L.M().a=cad>},
the application of $\sigma$ does not change that power.)
Indeed, $DM$ is free as a left
$\!DN\!$-module on the basis consisting of those elements $a\in M$ in
which the power of $z$ at the far left is $z^0;$
and by subtracting a right multiple of $p(z),$
we can reduce any element of that free module
to one in which the coefficient in $DN$ of each such
$a$ lies in a set of representatives of the residue classes
modulo $p(z)$ in $D[z],$ respectively, $D[z,z^{-1}],$
namely (in either case), the $\!D\!$-linear
combinations of $z^0,\dots,z^{d-1}.$

So we may assume below that the support each of $u_2,\dots,u_n,$
consists of elements of $M$ in whose normal forms
the factor $z^i$ at the left end satisfies $0\leq i<d.$
Moreover, by~\eqref{d.only_alpha_1}, none of
the monomials in the supports of $u_2,\dots,u_n$ are themselves powers
of $z;$ hence the number of $\!z\!$'s at their left ends
is unaffected by {\em right} multiplying by arbitrary elements of $M.$
This shows that
\begin{equation}\begin{minipage}[c]{35pc}\label{d.0...d-1}
In all monomials in the support of $\sum_{i=2}^n u_i\,v_i,$ the power
of $z$ occurring at the far left lies in the range $0,\dots,d-1.$
\end{minipage}\end{equation}

We will now get a contradiction, and complete the proof, by showing
that the same is not true of the term $u_1\,v_1=p(z)\,v_1.$
To do this, let $b$ be any monomial with no $z$ or $z^{-1}$ at
its far left, such that at least one term of the form $z^j\,b$
occurs in the support of $v_1.$
Thus, we can write the part of $v_1$ with support in $Nb$
as $q(z)\,b,$ for some $q(z)\neq 0$
in $DN$ (i.e., in $D[z]$ if $M=M_{(0)},$
or in $D[z,z^{-1}]$ if $M=M_{(1)}).$
Hence, the projection of $u_1 v_1$ onto $DNb$ will be $p(z) q(z)\,b.$
If $M=M_{(0)}$ we see that the degree of
$p(z) q(z)$ will be $\geq d.$
If $M=M_{(1)},$ the same will be true if $q(z)$ involves
at least one term in which $z$ has nonnegative exponent, while in
the contrary case, $p(z) q(z)$ will have a term in which $z$
has negative exponent.
In either case, comparing with~\eqref{d.0...d-1}
we see that $u_1 v_1$ cannot have zero sum with
$\sum_{i=2}^n u_i\,v_i.$
This contradiction completes the proof of the lemma.
\end{proof}

We shall now deduce from the above result the corresponding statement
for $\!\Z\!$-gradings, as promised in the Abstract.

\begin{proposition}\label{P.M>Z}
Let $M=M_{(0)}$ or $M_{(1)},$ let $\delta:G\to\Z\times\Z$ be as
in the preceding lemma,
and let $h:\Z\times\Z\to\Z$ be any nonzero homomorphism, so that
$h\delta:G\to\Z$ induces a $\!\Z\!$-grading on the monoid ring $DM.$
Then the $\!\Z\!$-graded ring
$DM$ is an $\!h\delta\!$-homogeneous semifir.
\end{proposition}

\begin{proof}
Let $\sum_{i=1}^n u_i v_i = 0$
be an $\!n\!$-term $\!h\delta\!$-homogeneous linear relation in $DM.$
Each of the $\!h\delta\!$-homogeneous elements $u_i$ and $v_i$
can be dissected into $\!\delta\!$-homogeneous summands,
with degrees in $\Z\times\Z$ which differ by members of
the cyclic group $\r{ker}(h),$ so we can picture each of these
sets of degrees as ``laid out along a line'' in $\Z\times\Z,$
parallel to the line given by $\r{ker}(h).$
Our idea will be to look at one end of each of these lines,
use the preceding lemma to trivialize the induced relation among the
terms at those ends, and thus ``eat away at'' the given relation until
it is trivialized.

But how do we know that the matrices $T$ that
we apply to eat away at one
end of our elements will not cause them to grow
at the other end, and hence survive indefinitely?
The key will be a fact that we saw in the preceding section,
but have not yet used in this one: that if we let
$\eta:\Z\times\Z\to\R$ be defined by
\begin{equation}\begin{minipage}[c]{35pc}\label{d.eta}
$\eta(r,s)\ =\ r+\tau\,s,$ where $\tau=(1+\sqrt 5)/2,$
\end{minipage}\end{equation}
then for every $a\in M,$ we have $\eta\delta(a)\geq 0$ in $\R.$
(This is the final statement of Lemma~\ref{L.delta}, though we are
now writing $\eta\delta$ for what we there called $\delta.)$
Hence if we eat away at the end of each element
corresponding to large values of $\eta\delta,$ we will hit a
barrier when those values try to become negative, and we will thus
successfully trivialize our relation.

Here are the details.

The kernel of $h:\Z\times\Z\to\Z$ is a cyclic subgroup of $\Z\times\Z;$
of its two generators, let $g$ be the one
such that $\eta(g)$ is positive.

Let us choose $\alpha_1,\dots,\alpha_n\in\Z\times\Z$ which are
preimages under $h$ of the $\!h\delta\!$-degrees of
$u_1,\dots,u_n,$ respectively, and $\beta$ which is a preimage of
the common $\!h\delta\!$-degree of $u_1 v_1,\dots,u_n v_n.$
The set of possible choices for each $\alpha_i,$
respectively for $\beta,$ is a coset of $g\Z;$ and by
the fact that
$\eta(g)\neq 0,$ we can choose our representatives of those cosets
so that each $\eta(\alpha_i)$ is $\leq 0,$ and
$\eta(\beta)$ is $\leq$ all the $\eta(\alpha_i).$
Thus, if for each $i,$ we write $u_i=\sum_j u_{ij},$ where
$u_{ij}$ is $\!\delta\!$-homogeneous of degree $\alpha_i+jg$
$(i\in\Z),$ then since every element of $M$ has
nonnegative $\!\eta\delta\!$-degree, while
the $\alpha_i$ are $\leq 0,$ the nonzero terms in this summation must
all have $j\geq 0.$
Similarly, $v_i$ is a sum of terms
$v_{ik}$ that are $\!\delta\!$-homogeneous of degrees
$-\alpha_i+\beta+kg$ with $k\geq 0.$

Now for each $i,$ let $j_i$ be the largest value such
that $u_{ij_i}\neq 0$ and $k_i$ the largest such that $v_{ik_i}\neq 0.$
Here we take $j_i$ or $k_i$ to be $-\infty$
if $u_i,$ respectively $v_i,$ is zero.
Let $\ell=\max_i(j_i+k_i).$
We shall use induction on $\ell,$ which is a nonnegative integer
as long as our relation is nontrivial.

By re-indexing, we may assume that the values
of $i$ for which $j_i+k_i=\ell$ are $1,\dots,m,$
where $1\leq m\leq n.$
Then under the $\!\delta\!$-grading of $DM,$
the component of $\sum_{i=1}^n u_i v_i$ of
$\!\Z\times\Z\!$-degree $\beta+\ell g$
is $\sum_{i=1}^m u_{ij_i} v_{ik_i}.$
Hence $\sum_{i=1}^m u_{ij_i} v_{ik_i}=0,$ which
is a $\!\delta\!$-homogeneous linear relation
(of degree $\beta+\ell g),$ so by the
preceding lemma, it can be trivialized.
If we extend the trivializing matrix by attaching
an $n{-}m\times n{-}m$ identity matrix, and apply this to
the original relation $\sum_{i=1}^n u_i v_i=0,$
the terms $u_i$ and $v_i$ with $i>m$ are not changed,
while for each $i\leq m,$ the terms $u_i$ and $v_i$ are modified
so as to decrease $j_i$ or $k_i.$
Thus the new relation has lower $\ell.$
In this way, we eventually reach $\ell<0,$ i.e.,
$\ell=-\infty,$ and have
thus trivialized our given relation.
\end{proof}

I had hoped to take this idea through one more iteration:
Given a relation $\sum u_i v_i=0$ holding in $DM,$
without any homogeneity assumption,
suppose one chooses an arbitrary
nonzero homomorphism $h:\Z\times\Z\to\Z,$
and grades the terms of this relation by $h\delta:M\to\Z.$
One can look at the components of the $u_i$ and $v_i$ on which
$h\delta$ assumes its largest value, take those $i$ where these two
values have the greatest sum, $\ell,$ note that the components
of those terms with $\!h\delta\!$-degree $\ell$ yield an
$\!h\delta\!$-homogeneous linear relation, trivialize it using
the above proposition, apply the trivializing matrix
(extended using an identity matrix) to the original relation,
and thus transform it into a relation with smaller $\ell;$ and
repeat this process indefinitely.

The trouble is that there is no reason why this
process should terminate.
The image of $M$ under $\delta$ is
$\{(r,s)\in\Z\times\Z\mid r+\tau s\geq 0\},$
i.e., the set of lattice points in a half-plane bounded by a
line of irrational slope; so no nonzero homomorphism $h:\Z\times\Z\to\Z$
assumes only positive values on that image.
Hence the $\ell$ of the above discussion, unlike the $\ell$ of the
proof of Proposition~\ref{P.M>Z}, need not be positive,
and there is no reason why the process should not decrease it
indefinitely.

I then thought, ``Aha!
Since we can continue this process indefinitely, it shows that
our relation is trivializable in the {\em completion} (in the
negative-degree direction) of $DM$ with respect to any such grading!''
But alas, when I tried to write out the argument, I found I could
not prove even that.
Namely, I can find no reason why the products of the matrices used
in the successive steps toward trivialization should converge
with respect to the grading.
But the idea is still tantalizing; perhaps someone can make it work.

Even if it does work, the semifir condition on a
completion is not very close to the condition that the non-completed
ring be a semifir.
For instance (working over a commutative base field $k$ for the
sake of familiarity), the group algebra on a free abelian group on two
generators, $k[x,\,x^{-1},\,y,\,y^{-1}]$ is far from being a
semifir -- ideals such as $(x-1,\,y-1)$ are not free -- but its
completion with respect to degree in $y,$ which we may write
$k[x,\,x^{-1}]((y,\,y^{-1})),$ is a fir.

A variant idea might be to try to complete $DM$ with respect
to the homomorphism $\eta\delta,$ which takes
values in the ordered group $\Z+\tau\Z\subseteq\R,$
either allowing infinite sums in the ``upward'' direction
(toward $+\infty)$ or the ``downward'' direction (toward $0).$
But completions with respect to non-discrete
ordered groups are, I suspect, a messy subject.

Finally, I have played with the idea of trying to trivialize
linear relations $\sum u_i v_i=0$ in $DM$ by successively
reducing, not the degrees of highest-degree terms, but
something like the number of lattice points in the convex hull of
the union of the images of the supports of the $u_i v_i$
in $\Z\times\Z.$
These numbers are invariant under the action of
$\sigma,$ and since they are nonnegative-integer valued,
a process that decreased them would terminate in finitely many steps.
But I don't see how to develop such a process.
It would require some technique that is applicable to the
$\!\Z\times\Z\!$-grading of $DM$ determined by the
homomorphism $M\to\Z\times\Z,$ but not to the $\!\Z\times\Z\!$-grading
of $k[x,\,x^{-1},\,y,\,y^{-1}]$ determined by the
{\em isomorphism} of the free abelian group on $x$ and $y$
with $\Z\times\Z$ (nor to the corresponding grading on the
monoid ring of any rank-$\!2\!$ submonoid of that free abelian group),
since $k[x,\,x^{-1},\,y,\,y^{-1}]$ is
not a $\!2\!$-fir (and similarly for monoid rings such as $k[x,\,y]).$

\section{A further observation on the above examples}\label{S.OK}

In the discussion by which we
motivated the definition of the monoid $M_{(0)},$
we started out with a relation $a\,b=b\,a\,g$ holding
in $G,$ saw what sort of elements of $G$
we had to throw into a monoid $M$ containing
$a,$ $b$ and $g$ so that~\eqref{d.ab=ca>}
would be satisfied for that relation, and then proceeded similarly
for an infinite sequence of further relations which that one led to.

Now in Ced\'{o}'s proof that every monoid $M$ such
that $DM$ is a $\!2\!$-fir must satisfy~\eqref{d.ab=cad>}, he
considers a relation $a\,b=c\,a\,d$ that holds in $M,$ notes that
in $DM$ this leads to the equation $a\,b-a\,d=c\,a\,d-a\,d,$
writes this as the $\!2\!$-term
linear relation $a\,(b-d)-(c-1)(a\,d)\ =\ 0,$
and examines what is required to trivialize that linear relation.
Let us restrict attention to the case $d=1,$ i.e., the case
of a relation $a\,b=c\,a$ as in~\eqref{d.ab=ca>},
so that Ced\'{o}'s linear relation takes the form
\begin{equation}\begin{minipage}[c]{35pc}\label{d.a(b-1)}
$a\,(b-1)-(c-1)\,a\ =\ 0.$
\end{minipage}\end{equation}

In the case of the relation $x\cdot y=y\cdot(xz),$
around which we built the monoid $M_{(0)},$ it is natural to wonder:
Can the corresponding relation in $DM$ actually be trivialized,
or do the succession of elements we have to adjoin to $M_{(0)}$
endlessly postpone the completion of this trivialization?

It turns out that the relation in question, namely,
$y\,(x\,z-1)-(x-1)\,y=0,$ is indeed trivializable in $DM.$
To see this, let us first apply $\sigma$ to that relation, getting
$y\,x\,y\,(y\,x\,z-1)-(y\,x-1)\,y\,x\,y=0,$
then reduce the term $y\,x\,z$ to normal form, getting
$y\,x\,y\,(x\,y-1)-(y\,x-1)\,y\,x\,y=0.$
The relation now has all its terms in the free
monoid ring $D\lang x,\,y\rang\md,$ which is a fir,
so the relation is indeed trivializable.
(If one works out the details,
the trivialization process, applied to the pair of
left-hand factors, $(y\,x\,y,\ -y\,x+1),$ begins by
adding to the first term the second term right-multiplied
by $y,$ getting $(y,\,-y\,x+1).$
One then adds to the second term the first right-multiplied by $x,$
getting $(y,\,1),$ and finally adds to the first
the second right-multiplied by $-y,$ giving $(0,1).$
The corresponding operations on the right-hand factors
of our relation necessarily yield a vector with
second entry zero, so the relation is indeed trivialized.)

So the scenario is not one where the steps of our construction of $M$
repeatedly postpone the trivialization of this relation;
rather, the first step renders that relation trivializable, but creates
{\em another} relation which must also be trivialized, and so on.

In fact, it is not hard to see that any relation of the
form $a(b-c)-(d-e)f=0$ holding in $DM_{(i)}$ with
$a,\dots,f\in M_{(i)}$ will be homogeneous with
respect to some $\!\Z\!$-grading $h\delta.$
For only two distinct elements of $M_{(i)}$ can
appear among $ab,$ $ac,$ $df,$ $ef,$ and we can find
a nonzero homomorphism $h:\Z\times\Z\to\Z$ such that
$h\delta$ has the same value on those two elements.
Hence the relation is $\!h\delta\!$-homogeneous,
so by Proposition~\ref{P.M>Z}, it is trivializable.

\section{Some monoids that misbehave}\label{S.M<G}
In this section, we turn back to the conditions on
monoids that we discussed in~\S\ref{S.conditions}, and give
examples of monoids that fail to satisfy one or another of them.
But I will start with a quick example of a monoid which does
satisfy those conditions -- but not in the obvious way.

\begin{example}\label{E.3<2}
Let $G=\lang x,\,y\rang\gp$ be the free group on two generators.
Then the submonoid $M=\lang x,\,x\,y,\,x\,y^2\rang\md\subseteq G$
generates $G$ as a group,
and $DM$ is a left and right fir, but the universal group
of $M$ is not $G,$ but the free group on three generators,
namely, on the images of $x,$ $x\,y$ and $x\,y^2.$
\end{example}

\begin{proof}
It is easily seen that $M$ is the free monoid on those three generators.
As a free monoid, it is one of the primordial examples of monoids
such that $DM$ is a right and left fir for all $D.$
But though it clearly generates the given group $G,$ its
{\em universal} group
is, necessarily, the free group on its free generators.
\end{proof}

Now for examples that don't satisfy all the conditions
of~\S\ref{S.conditions}.

\begin{example}\label{E.group}
Every group $G,$ regarded as a monoid,
satisfies~\eqref{d.M<G}-\eqref{d.ab=ca>''};
but if $G$ is not locally free, it will not satisfy~\eqref{d.G_l_f},
and so $DG$ will not be a semifir.
For $G=\Z\times\Z,$ that ring is not even a $\!2\!$-fir.
\end{example}

\begin{proof}
It is clear that every group satisfies~\eqref{d.M<G}-\eqref{d.a=cad>},
\eqref{d.ab=ca>'} and~\eqref{d.ab=ca>''}, and hence
by Lemma~\ref{L.ab=ca},~\eqref{d.ab=cad>},
\eqref{d.ba=cad>} and~\eqref{d.ab=ca>}.

(The last three conditions can also be obtained directly, by noting
that from the equation assumed in the
condition in question, and any choice of $n,$ one can solve for
the elements that make the conclusion hold.
For instance, given group elements satisfying $ab=ca$ as assumed
in~\eqref{d.ab=ca>}, and any natural number $n,$ we can solve $a=c^n e$
for $e,$ then solve $c=ef$ for $f,$ and verify that
$b=fe$ also holds.)

As mentioned earlier, \cite{WD+AS} shows that the
groups for which $DG$ is a semifir are those that are locally free,
i.e., satisfy~\eqref{d.G_l_f}.
For $G=\Z\times\Z,$ the ring $DG$ has the form
$D[x,\,x^{-1},\,y,\,y^{-1}],$ which is not a $\!2\!$-fir.
(For details, cf.\ next-to-last sentence
of proof of Corollary~\ref{C.commut} below.)
\end{proof}

At the opposite extreme from groups, we can find
monoids $M$ with no invertible elements $\neq 1$ which satisfy all
our conditions but~\eqref{d.G_l_f} and for which the rings
$DM$ are not $\!2\!$-firs.
Indeed, it is easy to see

\begin{example}[cf.\ Ced\'{o} {\cite[Example~1]{FC}}]\label{E.G>_0}
Let $G$ be any subgroup of the additive group of real numbers,
and $M=G^{\geq 0}$ the monoid of nonnegative elements of $G.$
Then $M$ satisfies~\eqref{d.M<G}-\eqref{d.ab=ca>''}; but
if $G$ contains two elements linearly independent over $\Q,$
it will not satisfy~\eqref{d.G_l_f}, and in fact $DG$
will not be a $\!2\!$-fir.\qed
\end{example}

The example of Ced\'{o} cited above is, up to isomorphism, the case
of Example~\ref{E.G>_0}
where $G$ is the additive subgroup of $\R$ generated by
$\{1,\tau\},$ where again, $\tau=(1+\sqrt 5)/2.$

Let us now give an example where, as in Example~\ref{E.3<2} above,
$M$ is a submonoid of a free group, but as in Examples~\ref{E.group}
and~\ref{E.G>_0}, its universal group is not free.

\begin{example}\label{E.vy=yu,vw=wu}
In the free group $G$ on $x$ and $y,$
let $M$ be the submonoid generated by
\begin{equation}\begin{minipage}[c]{35pc}\label{d.ywuv}
$y,\quad w = y\,x,\quad u = x^2,\quad v = y\, x^2 y^{-1}.$
\end{minipage}\end{equation}
Then the universal group of $M$ is the free product of
an infinite cyclic group and a free abelian group of rank~$2.$

This monoid $M$ satisfies~\eqref{d.M<G} \textup{(}and
hence~\eqref{d.cancel}\textup{)}, \eqref{d.a=cad>},
\eqref{d.ab=ca>'}, and~\eqref{d.ab=ca>''},
but not~\eqref{d.G_l_f}, \eqref{d.cap},
\eqref{d.ab=cad>}, \eqref{d.ba=cad>} or~\eqref{d.ab=ca>}.
\end{example}

\begin{proof}
Let us begin by showing that $M$ has the presentation
\begin{equation}\begin{minipage}[c]{35pc}\label{d.ywuv_pres}
$\lang y, w, u, v \mid vy = yu,\ vw = wu\rang\md.$
\end{minipage}\end{equation}

Clearly, the two relations of~\eqref{d.ywuv_pres} hold in $M.$
With their help we can reduce every element of the
monoid with presentation~\eqref{d.ywuv_pres} to a product
of $y, w, u, v$ in which
\begin{equation}\begin{minipage}[c]{35pc}\label{d.ywuv_prec}
no $v$ immediately precedes a $y$ or a $w.$
\end{minipage}\end{equation}
(Indeed, the substitutions $vy\mapsto yu$ and $vw\mapsto wu$
both decrease the number of occurrences of $v$ in a string, so
starting with any string, the process of repeatedly
applying these transformations terminates.)
Hence to prove that the natural homomorphism from~\eqref{d.ywuv_pres}
to the submonoid of $G$ generated by the elements~\eqref{d.ywuv} is an
isomorphism, it suffices to show that if we take a string $s$
made out of the symbols $y,$ $w,$ $u$ and $v$
satisfying~\eqref{d.ywuv_prec}, map it
into $G$ via the substitutions~\eqref{d.ywuv},
and write its image as a reduced group word in $x$
and $y,$ then the string $s$ can be recovered from that word.

I claim that for $s$ satisfying~\eqref{d.ywuv_prec},
every maximal nonempty string of $\!v\!$'s in $s,$
say $v^n,$ yields a string $y\, x^{2n} y^{-1}$ in the reduced
form of the image element, in which the $y$ on the left and
the $y^{-1}$ on the right are not cancelled by any
other terms.
This is because our $v^n$ is not preceded by a $v,$
or followed by a $y,$ $w$ or $v.$
Since factors $y^{-1}$ in our reduced expression arise only from
these substrings $v^n,$ we can locate all such substrings,
and thus write the image
of $s$ as a product of strings $v^n$ alternating with words that
are products of $y,$ $w=y\,x,$ and $u=x^2.$
But from any product of $y,$ $y\,x,$ and $x^2$ we can easily
recover the factors (using the observation that a $y$ in
the product will come from a factor $y$ if the number of $\!x\!$'s
that follows it is even, but from a factor $y\,x$ if that
number is odd).
Hence we can reconstruct the reduced word~$s.$

It follows that~\eqref{d.ywuv_pres} is a presentation of $M,$
and hence that its universal group has the group-presentation
$\lang y, w, u, v \mid vy = yu,\ vw = wu\rang\gp.$
Now the first relation in that presentation can be written
as $v=yuy^{-1},$ and we may use this to eliminate $v$ from
the presentation.
This transforms the other relation into $yuy^{-1}w = wu.$
If we simplify this relation by making a change
of generators of our group, letting $z=y^{-1}w$ and
eliminating $w$ in favor of $z,$ our presentation becomes
\begin{equation}\begin{minipage}[c]{35pc}\label{d.yzu_pres}
$\lang y,\,z,\,u \mid uz = zu\rang\gp.$
\end{minipage}\end{equation}
Clearly, this group is, as claimed, the coproduct of the infinite
cyclic group on $y$ and the free abelian
group on $u$ and $z.$
Since in a free group, every abelian subgroup is
cyclic,~\eqref{d.yzu_pres} is not free, hence, being
finitely generated, it is not locally
free, so $M$ does not satisfy~\eqref{d.G_l_f}.
(What happens when we apply the natural map of this group into $G$?
We find that $u$ falls together with~$z^2.)$

On the other hand, since $M$ is given as a submonoid
of a group, it satisfies~\eqref{d.M<G}.
The properties~\eqref{d.a=cad>},
\eqref{d.ab=ca>'} and~\eqref{d.ab=ca>''} can be seen by
applying to the equations assumed in those conditions
the homomorphism $G\to\Z$ taking $x$ and $y$ to $1,$
and noting that each of the generators of $M$
has positive image under that map.

Finally, either of the relations
in~\eqref{d.ywuv_pres} is a counterexample to conditions~\eqref{d.cap},
\eqref{d.ab=cad>} (with $d=1),$ \eqref{d.ba=cad>} (with $c=1),$
and \eqref{d.ab=ca>}.
\end{proof}

V.\,Maltcev (personal communication) has pointed out that there
exist examples similar to the above in which
$M$ is wholly contained in the free monoid on the given generators.
Indeed, one can get such an example by applying to
Example~\ref{E.vy=yu,vw=wu} the automorphism of
$G$ that fixes $y,$ and carries $x$ to $xy.$

The fact that Example~\ref{E.vy=yu,vw=wu}
satisfies~\eqref{d.ab=ca>'} and~\eqref{d.ab=ca>''}
without~\eqref{d.ab=ca>} shows the need for the
hypothesis~\eqref{d.cap} in Lemma~\ref{L.ab=ca}.
I don't know a similar example which satisfies~\eqref{d.cap}.
More precisely, I don't know the answer to

\begin{question}\label{Q.=>G_l_f}
If $M$ is a submonoid of a locally free \textup{(}respectively,
a free\textup{)} group, and satisfies
\eqref{d.cap}, must the universal group of $M$ be locally free
\textup{(}respectively, free\textup{)}?

If not, does the answer change if we
assume $G$ also satisfies~\eqref{d.a=cad>} and~\eqref{d.ab=ca>}?
\end{question}

Let us now look at a couple of submonoids $M$ of free groups $G$ which
do have  $G$ as their universal groups, and so
satisfy~\eqref{d.G_l_f},~\eqref{d.M<G}
and~\eqref{d.cancel}, but which fail to satisfy
some of our other conditions.
An easy one is

\begin{example}\label{E.x^2,x^3}
Let $G$ be the free group on one
generator $x,$ and $M$ the submonoid generated by $x^2$ and $x^3.$
Then $M$ satisfies~\eqref{d.G_l_f}, \eqref{d.M<G}
\textup{(}hence~\eqref{d.cancel}\textup{)},
\eqref{d.a=cad>}, \eqref{d.ab=ca>'}, and~\eqref{d.ab=ca>''},
but not~\eqref{d.cap} or~\eqref{d.ab=ca>}.
\end{example}

\begin{proof}
Conditions~\eqref{d.G_l_f} and \eqref{d.M<G}
are easily verified, the universal
group being $G$ itself, and the remaining positive results
are seen by looking at degree in $x.$
On the other hand, the relation $x^2\cdot x^3=x^3\cdot x^2,$ where
neither $x^2$ not $x^3$ right divides the other in $M,$ gives the two
negative assertions.
\end{proof}

To get a monoid not satisfying~\eqref{d.a=cad>},
we can work in the free group $G$
on generators $x$ and $y,$ take $a=x,$ $c=y,$ and ``solve''
for the $d$ in the equation $a=cad.$
The arguments establishing the properties of this
monoid are similar to those used in previous examples, so I
will leave many details to the reader.

\begin{example}\label{E.x'y'x}
In the free group $G$ on $x$ and $y,$
let $M$ be the submonoid generated by
\begin{equation}\begin{minipage}[c]{35pc}\label{d.xyw}
$x,\quad y,\quad w = x^{-1} y^{-1} x.$
\end{minipage}\end{equation}
Then the universal group of $M$ is $G,$ with the inclusion
as the universal map; and $M$ satisfies~\eqref{d.G_l_f},
\eqref{d.M<G} \textup{(}hence~\eqref{d.cancel}\textup{)},
\eqref{d.cap}, \eqref{d.ab=ca>}, \eqref{d.ab=ca>'}
and~\eqref{d.ab=ca>''}, but not~\eqref{d.a=cad>} \textup{(}and
hence not~\eqref{d.ab=cad>} or~\eqref{d.ba=cad>}\textup{)}.
\end{example}

\begin{proof}[Sketch of proof]
Noting that $w^n=x^{-1}y^{-n} x,$ one verifies that
the reduced expressions in $G$ for the elements of $M$
are the strings in which every $x^{-1}$ is followed immediately
by a $y^{-1},$ and every maximal nonempty string of $\!y^{-1}\!$'s is
followed immediately by an $x.$
Noting that $yxw=x,$ it is not hard to deduce that
a normal form for $M$ is given by all words in
$x,$ $y,$ $w$ containing no substring $y\,x\,w,$ and
from this to get the presentation
\begin{equation}\begin{minipage}[c]{35pc}\label{d.x=yxw}
$M\ =\ \lang x,\,y,\,w\mid y\,x\,w=x\rang\md.$
\end{minipage}\end{equation}
Hence the universal group of $M$ is
$\lang x,\,y,\,w\mid x=yxw\rang\gp.$
We can use the one relation to eliminate $w,$ showing that
this group is $G,$ and
establishing~\eqref{d.G_l_f} and \eqref{d.M<G}.

The proof of~\eqref{d.cap} is lengthy, but can be made somewhat
shorter by first noting that the automorphism $\theta$
of $G$ which fixes $y$ and carries $x$ to $y\,x$ also fixes $w,$
and hence takes $M$ into itself.
Moreover, $\theta^{-1}$
also carries $M$ into itself, taking $x$ to $y^{-1}x=x\,w;$
so $\theta$ yields an automorphism of $M.$
By applying a sufficiently high power of $\theta^{-1}$ to
any relation $a\,c=b\,d$ in $M,$ we can assume that none of the
normal form expressions for $a,$ $c,$ $b$ and $d$ involve
the sequence $y\,x,$ and that in all of them, every $x$ is
immediately followed by a $w.$
The conclusion of~\eqref{d.cap} is immediate unless
at least one of the pairs of factors
in that relation, say $a\,c,$ undergoes a nontrivial
reduction to normal form when multiplied together,
in which case we find that we can write
\begin{equation}\begin{minipage}[b]{35pc}\label{d.a=...}
$a=a'y^n,\quad c=x\,w^n\,c',$
\end{minipage}\end{equation}
where $n\geq 1,$ and where either $a'$ does not end with a $y,$ or $c'$
does not begin with a $w.$
Thus, $a\,c$ has the normal form $a'\,x\,c'.$

We then consider how the normal form of $a\,c$ can
equal that of $b\,d.$
If $b\,d$ requires no reduction to bring it to normal form,
then since $b$ left divides the element whose normal
form is $a'x\,c',$ it either left divides $a',$ in which case it left
divides $a,$ or it is a right multiple of $a'x,$ in which case
it is a right multiple of $a=a'y^n,$ since the relation
$\theta^n(x)=y^n x$ shows that $x$ is a right multiple of $y^n.$
If $b\,d$ does require reduction, we have $b=b'y^m,$
$d=x\,w^m\,b',$ with normal form $b'\,x\,d'.$
We then find that either $b'\,x$ left-divides $a',$
or $a'\,x$ left-divides $b',$ or $a'=b'.$
In the first case, $b$ left divides $a,$
in the next, $a$ left divides $b,$
while in the last, we have one or the other, depending
on whether $m\leq n$ or $m\geq n.$

To get~\eqref{d.ab=ca>}-\eqref{d.ab=ca>''}, one first
deduces from the arguments showing
the equivalence of these three conditions (Lemma~\ref{L.ab=ca}) that
in a relation $a\,b=c\,a$ contradicting~\eqref{d.ab=ca>},
the element $a$ would have to be
{\em both} left divisible by arbitrarily large powers of $c$ and
right divisible by arbitrarily large powers of $b.$
However, one can verify
that the only elements $u$ of $M$ by which an element can be infinitely
left divisible are the nonnegative powers of $y,$ and the
only ones by which an element can be infinitely right
divisible are the nonnegative powers of $w.$
(Key ideas: by symmetry of~\eqref{d.x=yxw}, it suffices
to show the former statement.
By considering degree in $x,$ we see that the only possibilities
for such a $u$ are words in $y$ and $w.$
But one also sees that reduction of a word to its normal
form does not alter the substring to the left of the
leftmost $x$ other than by changing the number of $\!y\!$'s
at its right end.
This eliminates all possible $u$ other than powers of $y.)$
Hence in our relation $a\,b=c\,a,$
$c$ would have to be a nonnegative power
of $y$ and $b$ a nonnegative power of $x^{-1}y^{-1}x.$
But this, together with the condition
in~\eqref{d.ab=ca>} that $b$ and $c$ not both equal $1,$ is seen to
yield a contradiction regarding the total degree in $y$
(i.e., the image under the homomorphism $G\to\Z$
taking $y$ to $1$ and $x$ to $0)$
of the common value of the two sides.

Finally, to show that $M$ does {\em not} satisfy~\eqref{d.a=cad>},
we of course use the relation $x=y\,x\,w,$ which we built
this monoid to satisfy.
We only need to verify that $x$ is not invertible, and this
is clear by looking at degrees in $x.$
(In fact, $M$ has no invertible elements other than $1,$
since the elements of degree $0$ in $x$
must, as noted, lie in the monoid generated by $y$ and $w,$
and that monoid is free on those generators.)
\end{proof}

One would similarly hope to get a monoid $M$ for which~\eqref{d.ab=ca>}
fails but~\eqref{d.G_l_f}-\eqref{d.a=cad>} hold by looking
inside the free group on $x$ and $y,$ letting $a=x$ and $b=y,$
and solving the equation $ab=ca$ for $c.$
Unfortunately, the close relation
between~\eqref{d.ab=ca>} and~\eqref{d.cap} shows itself
again: in the resulting monoid, the equation designed to
invalidate~\eqref{d.ab=ca>} also invalidates~\eqref{d.cap}.
However, here is an example, due to Ced\'{o}, of
a monoid satisfying~\eqref{d.G_l_f}-\eqref{d.a=cad>}
but not~\eqref{d.ab=ca>}.
It is constructed, like the earlier example $M_{(0)},$ as a direct
limit of copies of the monoid we have called $M_0,$ but the
limit is now taken with respect to a different automorphism of $G.$

\begin{example}[{\cite[Example~4]{FC}}]\label{E.ndots=0}
Let $G$ be the free group on $x$ and $y,$
and as in \S\S\ref{S.C_eg}-\ref{S.OK}, let
$M_0$ be the submonoid of $G$ generated by
$x,$ $y,$ and $z=x^{-1}y^{-1}x\,y.$
Let $\varphi$ be the
automorphism of $G$ which fixes $x$ and carries $y$ to $x\,y,$
which we see carries $M_0$ into itself, and let
\begin{equation}\begin{minipage}[c]{35pc}\label{d.M_non_ab=ca}
$M\ =$ the set of $a\in G$ such that $\varphi^n(a)\in M_0$ for
some \textup{(}hence, for all sufficiently large\textup{)} $n.$
\end{minipage}\end{equation}

Then $M$ satisfies~\eqref{d.G_l_f}-\eqref{d.a=cad>}, and
its inclusion in $G$ gives its universal group;
but $M$ does not \mbox{satisfy~\eqref{d.ab=ca>}-\eqref{d.ab=ca>''}.}
\end{example}

\begin{proof}[Sketch of proof]
As we saw in Lemma~\ref{L.M_0}, $M_0$ has the presentation
$\lang x,\,y,\,z\mid y\,x\,z=x\,y\rang\md,$ has a normal form given
by the words having no subwords $y\,x\,z,$
and satisfies~\eqref{d.G_l_f}-\eqref{d.cancel}.
It follows that the direct limit $M$ of copies of $M_0$
likewise satisfies~\eqref{d.G_l_f}-\eqref{d.cancel}.

As in the proof of Lemma~\ref{L.M().cap},
to obtain~\eqref{d.cap}, we will consider a relation $a\,b=c\,d$
holding in $M,$ and by applying $\varphi$ sufficiently many
times, assume without loss of generality that $a,$ $b,$ $c$
and $d$ all lie in $M_0,$ regard them as written in normal form,
and consider possible reductions that occur when $a\cdot b$
and $c\cdot d$ are reduced to normal form.
But because of the difference between the automorphism $\sigma$
we used then and the $\varphi$ in the construction of
this monoid, we will not be able to use further applications of that
automorphism to put additional useful
restrictions on the normal forms of our factors.

We find that if, on multiplying
together the normal forms for $a$ and $b,$ a reduction
occurs, these elements must have either the forms $a=a'\,y$
and $b=(x\,z)^m\,b',$ or $a=a'\,y\,x$ and $b=z\,(x\,z)^{m-1}\,b',$
where in either case $m>0,$ $a',\,b'\in M_0,$ and $b'$ has normal form
not beginning $x\,z;$ and that the normal form of
the product is in each case $a'\,x^m\,y\,b'.$
The analogous observations hold for $c$ and $d.$
As in the proof of Lemma~\ref{L.M().cap},
we can assume that at least the calculation
of $a\cdot b$ does involve such a reduction, and subdivide into cases
according to whether the calculation of $c\cdot d$ does or
does not, and how far the left factor $c'$ (if the calculation
does involve reduction) or $c$ (if it does not) extends into
the product $a'\,x^m\,y\,b'.$

When the multiplication of $c\cdot d$ does not involve reduction,
$c$ must be an initial substring of the
normal form $a'\,x^m\,y\,b'$ of $a\cdot b.$
Here, the case in which it is not obvious that one of $a$ or $c$
left-divides the other is when $c$ has the form $a'\,x^r$ $(r\leq m).$
In that case, we right multiply $c=a'\,x^r$ by $x^{-r}y\in M,$
and thus see
that $c M$ contains $a'\,y,$ and hence $a'\,y\,x,$ one of which is $a.$
When the multiplication of $c\cdot d$ does involve reduction,
so that the product has the form $c'\,x^n\,y\,d',$ then after
possibly interchanging the roles of $a$ and $c,$ the nontrivial case
is where the normal forms of the two equal products align in such
a way that $a'=c'\,x^r$ with $0<r\leq n.$
In that case, on right-multiplying  $c=c'\,y,$
respectively $c=c'\,y\,x,$ by $(x\,z)^r,$
respectively $z\,(x\,z)^{r-1},$ and then if necessary by $x,$
we again get $a\in c M.$

Condition~\eqref{d.a=cad>} reduces to that condition
on the monoid $M_0,$
and this was obtained in the proof of Lemma~\ref{L.M().a=cad>}.

Finally, the relation
$y\cdot(xz)=x\cdot y$ witnesses the failure of~\eqref{d.ab=ca>'},
since $y$ is infinitely left divisible by $x$ in~$M.$
\end{proof}

\section{A result of Dicks and Menal, and division-closed submonoids}\label{S.div_cl}
Dicks and Menal prove in \cite{WD+PM} that if $G$ is a
group such that $DG$ is an $\!n\!$-fir, then every $\!n\!$-generator
subgroup of $G$ is free.
Their proof generalizes easily to give the following
result, which we shall use in the next section.
Recall that a ring $R$ is called {\em left $\!n\!$-hereditary} if
every $\!n\!$-generator left ideal is projective as a left
$\!R\!$-module.

\begin{theorem}[after Dicks and Menal \cite{WD+PM}]\label{T.n-gen_sgp}
Let $M$ be a monoid, and $n$ a positive integer such that
$DM$ is an $\!n\!$-fir, or more generally, is a left
$\!n\!$-hereditary $\!1\!$-fir.
Then

\textup{(i)} For every subgroup $H$
of the group of invertible elements of $M,$
$DH$ is again a left $\!n\!$-hereditary $\!1\!$-fir.

\textup{(ii)} Every $\!n\!$-generator subgroup $H$
of the group of invertible elements of $M$ is free.
\end{theorem}

\begin{proof}
The fact that $DM$ is a $\!1\!$-fir, i.e., a ring without zero divisors,
implies, on the one hand, that $M$ satisfies the cancellation
condition~\eqref{d.cancel}, and on the other hand, that
its group of invertible elements has no elements of finite order.
Now given $H$ as in~(i),
since $M$ is cancellative, the right cosets
of $H$ in $M,$ i.e., the orbits in $M$ of
the left action of $H,$ are each isomorphic to $H.$
Of course, the same is true of the left cosets as right orbits; but
this will not be immediately useful to us, except for the
consequence that the orbit $H$ itself, and its
complement within $M,$ are closed under both actions.
Combining the preceding observation with this fact, we get
\begin{equation}\begin{minipage}[c]{35pc}\label{d.(+)d}
$DM$ is free as a left $\!DH\!$-module, and $DH$ is a direct summand
therein as $\!DH\!$-bimodules.
\end{minipage}\end{equation}

Now let $A\subseteq DH$ be any $\!n\!$-element subset.
Since $DM$ is $\!n\!$-hereditary, $(DM)A$ is
projective as a left $\!DM\!$-module,
so as $DM$ is left free over $DH,$
$(DM)A$ is also projective as a left $\!DH\!$-module.
Since, by~\eqref{d.(+)d}, the projection of $DM$ onto $DH$ respects
right multiplication by the elements of $A,$ it
carries $(DM)A$ onto $(DH)A,$
so $(DH)A$ is a $\!DH\!$-direct summand in the projective
left $\!DH\!$-module $(DM)A.$
Hence it is projective
as a left $\!DH\!$-module, giving the conclusion of~(i).

To get~(ii), note that if $h_1,\dots,h_n$ generate $H,$ then
$h_1-1,\dots,h_n-1$ generate the augmentation ideal $I$ of $DH$
as a left ideal, so by~(i), that ideal is projective as
a left $\!DH\!$-module.
Thus $H$ is a finitely generated group
without elements of finite order,
whose augmentation ideal is projective as a left module.
The Swan-Stallings Theorem \cite[Theorem~A]{DEC} says that any such
group $H$ is free.
\end{proof}

Can we get a similar result with $H$ replaced by a more general
submonoid $N$ of $M$?
In general, for a submonoid $N$ of a monoid $M,$
there is no nice analog of the coset decomposition used above.
But for certain submonoids of a
monoid $M$ that satisfies~\eqref{d.cancel} and~\eqref{d.cap},
we can get such a decomposition.

\begin{definition}\label{D.closed}
A submonoid $N$ of a monoid $M$ will be called
{\em left division-closed} if for elements $a,\,b\in M$
with $a\,b\in N,$ we have $a\in N\implies b\in N.$
It will be called
{\em right division-closed} if the reverse implication holds.

If $M$ is a monoid satisfying~\eqref{d.cancel} and~\eqref{d.cap},
and $N$ a left division-closed submonoid of $M,$
we shall call elements $x,\,y\in M$ {\em left
$\!N\!$-equivalent} if they lie in a common principal
right coset $N z$ $(z\in M).$
\textup{(}This terminology will be justified by point \textup{(ii)}
of the next result.\textup{)}
\end{definition}

\begin{lemma}\label{L.N-equiv}
Let $M$ be a monoid satisfying~\eqref{d.cancel} and~\eqref{d.cap},
and $N$ a left division-closed submonoid of $M.$
Then

\textup{(i)} $N$ also satisfies~\eqref{d.cancel} and~\eqref{d.cap}.

\textup{(ii)}  Left $\!N\!$-equivalence is an
equivalence relation on $M.$

\textup{(iii)}  Each left $\!N\!$-equivalence class is a
directed union of principal right cosets $N x;$

\textup{(iv)}  $N$ is itself a left $\!N\!$-equivalence class in $M.$

\textup{(v)} If two right cosets $N x$ and $N y$ have
nonempty intersection, then one of them contains the other;
so in particular, $x$ and $y$ are left $\!N\!$-equivalent.
\end{lemma}

\begin{proof}
To see (i), observe that $N$ inherits~\eqref{d.cancel} from $M,$
and also~\eqref{d.cap}, given that it is left division-closed.

Condition~\eqref{d.cap} and the definition of left division-closed
are also easily seen to imply~(v).

In~(ii), reflexivity and symmetry are clear.
To see transitivity, suppose $x$ left
$\!N\!$-equivalent to $y,$ and $y$ to $z.$
Thus, we have $u,\,v\in M$ such that $N u$ contains the first two
of these elements, and $N v$ the last two.
Since the intersection of $N u$ and $N v$ contains
$y,$~(v) says that one of those two cosets contains the
other, and hence contains both $x$ and $z,$
proving these left $\!N\!$-equivalent.

The definition of left $\!N\!$-equivalence, and~(v), together show
that any left $\!N\!$-equivalence class is a directed union
of principal right cosets, i.e.,~(iii).
We shall prove~(iv) by
showing that any $x\in M$ equivalent to $1$ is contained in $N.$
Say $1$ and $x$ are both contained in a right coset $N y.$
Since $1\in N y,$ we can write $1=a y$ with $a\in N.$
Hence left division-closure of $N$ implies $y\in N,$ so
$N\supseteq N y\ni x,$ as desired.
\end{proof}

Let us now look at some consequences for monoid rings.
These considerations will not require the
base ring to be a division ring $D,$ so we
give the next result without that assumption.

\begin{lemma}\label{L.DN<DM}
Let $M$ be a monoid satisfying~\eqref{d.cancel} and~\eqref{d.cap},
$N$ a left division-closed submonoid of $M,$ and $R$ any ring.
Then

\textup{(i)} The monoid ring $RM$ is flat as a left $\!RN\!$-module.

\textup{(ii)} The projection map
$RM\to RN$ taking $\sum_{a\in M} c_a\,a$ to $\sum_{a\in N} c_a\,a$
is a left $\!RN\!$-module homomorphism.

If, in addition to being left division-closed,
$N$ is right division-closed, and $A$ is a subset of $RN,$ then

\textup{(iii)} If $(RM)A$ is flat as a left $\!RM\!$-module,
then $(RN)A$ is flat as a left $\!RN\!$-module.

\textup{(iv)} If $A$ is finite and $(RM)A$ is projective as a left
$\!RM\!$-module, then $(RN)A$ is projective as a left $\!RN\!$-module.
\end{lemma}

\begin{proof}
By Lemma~\ref{L.N-equiv}\,(ii)-(iii), $RM$ is the direct sum,
over the left $\!N\!$-equivalence classes $E\subseteq M,$
of the $\!R\!$-submodules $RE,$ and each of
those is a left $\!RN\!$-submodule, isomorphic to a
direct limit of (rank-$\!1)\!$ free $\!RN\!$-modules, hence
flat over $RN.$
This gives~(i).
By Lemma~\ref{L.N-equiv}\,(iv),
$RN$ is one of the direct summands, and the module of
elements with support in the complement of $N$ is the sum
of the others, so we have~(ii).

The additional hypothesis of~(iii) and~(iv), that $N$ is
{\em right} division-closed, says that the complement of $N$ in $M$
is closed under right multiplication by elements of $N.$
Thus $RN$ (because it is a ring), and its complementary summand
(by the above condition) are each
closed under right multiplication by elements of $RN,$ hence in
particular, by elements of $A;$ so the left $\!RM\!$-module
$(RM)A$ is the direct sum of an $\!RN\!$-submodule
lying in $RN$ and one lying
in that complementary left $\!RN\!$-submodule.

If $(RM)A$ is flat as a left $\!RM\!$-module, then
it is a direct limit of free $\!RM\!$-modules, and by~(i) each of
these is flat as an $\!RN\!$-module, so $(RM)A$ is flat as an
$\!RN\!$-module.
Hence so is its direct summand $(RN)A,$ proving~(iii).

To prove~(iv) we shall use J{\o}ndrup's Lemma \cite[Lemma~1.9]{WD+AS},
which says that any finitely generated flat module over a ring,
which on extension of scalars to some overring gives a projective
module over that ring, must already be projective over the given ring.
Now if $(RM)A$ is a projective $\!RM\!$-module,
it is in particular flat, hence by~(iii), $(RN)A$ is flat.
To see what happens when we extend scalars to $RM,$
note that since $N$ is assumed right as well as left
division-closed, the left-right dual of~(i) says that
$RM$ is flat as a right $\!RN\!$-module.
Hence when we left-tensor the
inclusion of $\!RN\!$-modules $(RN)A\hookrightarrow RN$ with $RM,$
we get an embedding $RM\otimes_{RN} (RN)A\hookrightarrow RM.$
This says that $RM\otimes_{RN} (RN)A$ is naturally
isomorphic to $(RM)A,$ so assuming the latter projective as a left
$\!RM\!$-module, the hypotheses of J{\o}ndrup's Lemma are
satisfied, and we get the desired projectivity of $(RN)A.$

(Warren Dicks has pointed out to me that one can get~(iv) directly
from~(i) and~(ii), using the fact that as an $\!RN\!$-module,
$RM$ is not merely flat but is a directed union of free modules.
One deduces that $(RM)A,$ since it is projective, and hence $(RN)A,$
as a direct summand therein, are direct summands in such
directed unions; hence the latter, being a finitely
generated $\!RN\!$-module, must be a direct summand
in some free submodule in the limit construction, hence projective.
Note that this argument still needs
the right division-closed condition to show that
$(RN)A$ is the image of $(RM)A$ under the projection $RM\to RN.)$
\end{proof}

We can now partially generalize of the proof of
Theorem~\ref{T.n-gen_sgp}, to get the following result, which will
also be used in the next section.

\begin{theorem}\label{T.n-hered}
Let $M$ be a monoid and $n$ a positive integer such that
$DM$ is an $\!n\!$-fir; or if $n>2,$ more generally,
such that $DM$ is a $\!2\!$-fir which is left $\!n\!$-hereditary.
Then for every left and right division-closed submonoid $N\subseteq M,$
the ring $DN$ is a $\!1\!$-fir and is left $\!n\!$-hereditary.
\end{theorem}

\begin{proof}
Whatever the value of $n,$ our hypotheses insure that $DM$ is a
$\!1\!$-fir, i.e., an integral domain; hence so is any subring thereof.
In particular, since a $\!1\!$-fir is $\!1\!$-hereditary,
this gives the required conclusions in the case $n=1.$

If $n\geq 2,$ then since $DM$ is a $\!2\!$-fir,
$M$ satisfies~\eqref{d.cancel} and~\eqref{d.cap}.
Hence we can apply Lemma~\ref{L.DN<DM}(iv) to $\!n\!$-generator
left ideals $(DN)A\subseteq DN,$ and conclude that they are
projective, making $DN$ $\!n\!$-hereditary, as claimed.
\end{proof}

We would, of course, like to know more.

\begin{question}\label{Q.DG}
If $M$ is a monoid such that $DM$ is an $\!n\!$-fir,
must $DN$ be an $\!n\!$-fir for every
right and left division-closed submonoid $N$ of $M$?
Must this at least be true if $N$ is a subgroup of
the group of invertible elements of $M$?
If it is precisely the group of such invertible elements?
\end{question}

(The referee has provided an answer to the second sentence of
the above question, which will be noted in the published version.)

Another question suggested by the methods we have used is the following.

\begin{question}\label{Q.flat}
What can one say about a monoid $M$ such that the augmentation
ideal of $DM,$ i.e., the ideal
$I$ generated by all elements $a-1$ $(a\in M),$
is flat as a left module, under various hypothesis on $M$?
\end{question}

Warren Dicks (personal communication) notes that results on this
question would be of interest even for $M$ a $\!2\!$-generator group.

One situation where we can get such results is if $M$ is a finitely
generated monoid such that $DM$ is embeddable in a division ring $D'.$
For in this situation, $D'\otimes_{DM} I$ is clearly free as
a $\!D'\!$-vector-space, hence by J{\o}ndrup's Lemma, if $I$
is flat, it is projective as a left $\!DM\!$-module.
In particular, if $M$ is a {\em group} admitting a two-sided
invariant ordering, it is known that $DM$ is embeddable
in a division ring, and $M$ has no elements of finite
order; so if such an $M$ is finitely generated
and its augmentation ideal is flat, we can again apply
the Swan-Stallings Theorem and conclude that $M$ is a free group.

(It is an open question whether for every group $M$ admitting a
{\em one-sided}
invariant ordering, $DM$ is embeddable in a division ring.
A positive answer to that question would clearly allow us to strengthen
the above observation.)

Returning to the properties of a general left division-closed submonoid
$N$ of a monoid $M$ satisfying~\eqref{d.cancel} and~\eqref{d.cap},
note that in contrast to the case of right cosets of a subgroup,
the left $\!N\!$-equivalences classes in $M$
are not in general mutually isomorphic as $\!N\!$-sets.
(For example, if $G$ is the free group on $x$ and $y,$ and $M$ the
submonoid generated by $x$ and all elements $x^{-n}y$ $(n\geq 0),$
let $N=\lang x\rang\md,$ and compare the $\!N\!$-equivalence
classes of $1$ and~$y.)$

The next result shows a bit of structure that one can extract from
this non-isomorphism.
(We will not use it below.)

\begin{lemma}\label{L.non-emb}
Let $M$ be a monoid satisfying~\eqref{d.cancel} and~\eqref{d.cap},
and $N$ a left division-closed submonoid of $M.$
Then

\textup{(i)} The set of $a\in M$ such that the left
$\!N\!$-equivalence class of $a$
is not embeddable as an $\!N\!$-set in $N$ forms a right ideal of $M.$

\textup{(ii)}
If $N$ has the stronger property that $ab\in N\implies a,\,b\in N,$
then the complementary set, of those $a\in M$ such that the left
$\!N\!$-equivalence class of $a$
{\em is} embeddable as an $\!N\!$-set in $N,$ forms a submonoid of~$M.$
\end{lemma}

\begin{proof}
To get (i), note that
the definition of left $\!N\!$-equivalence shows that
if $a$ and $b$ are left $\!N\!$-equivalent, then for any
$c\in M,$ the elements $a\,c$ and $b\,c$ will also be.
Hence the left $\!N\!$-equivalence class of any element $a$ is mapped
by right multiplication by $c$ into the left $\!N\!$-equivalence
class of $a\,c.$
Since $M$ is assumed cancellative, this
map is an embedding; so if the equivalence class of $a$ is
not embeddable in $N,$ neither is that of $a\,c.$

Before proving (ii), let us note that
from~\eqref{d.cancel} and~\eqref{d.cap}, it is not hard to
deduce that the left $\!N\!$-set structure of
the left $\!N\!$-equivalence class of an element of $M$ is determined
by the set of elements of $N$ that left divide that element.

Now assume the hypothesis of (ii),
and let $N'$ be the set of $a$ such that the
left $\!N\!$-equivalence class of $a$ is embeddable in $N.$
Clearly, $1\in N'.$
Suppose $x,\,y\in N'.$
The set of elements of $N$ that left divide $x\,y$
contains those that left divide $x.$
If it contains nothing more, then by the preceding observation,
the equivalence
classes of $x\,y$ and $x$ are isomorphic, showing that $x\,y\in N'.$
If, on the other hand, $x\,y,$ but not $x,$ is
left divisible by some $a\in N,$ let $x\,y=a\,z.$
Then~\eqref{d.cap} and the fact that $x$ is not left divisible by $a$
show that $a=x\,b$ for some $b\in M.$
This relation and our hypothesis on $N$ tell us that $x\in N,$
so $x\,y\in N y,$ so the left $\!N\!$-equivalence class of $x\,y$ is
that of $y,$ which by assumption is
embeddable as a left $\!N\!$-set in~$N.$
\end{proof}

I wonder whether some much weaker hypothesis than
$ab\in N\implies a,\,b\in N$ could be used in~(ii) above.
This is suggested by the fact that
when $N$ is a proper subgroup of the group of invertible
elements of $M,$ the conclusion of~(ii) holds (since then $N'$
is all of $M);$ but that hypothesis does not.

\section{Commuting elements}\label{S.cm}

It is well-known that every subgroup of a free group is free.

Hence if two elements of a free group commute, the subgroup they
generate must be cyclic.
We can say more: the centralizer
of every nonidentity element $a$ of a free group is cyclic.
For the center of that centralizer contains $a,$ and the only
free group with nontrivial center is the free group on one generator.
From these facts, one can deduce similar facts about
commuting elements and centralizers in direct limits
of free groups, i.e., groups $G$ such that $DG$ is a semifir;
and using slightly more elaborate arguments,
in groups $G$ such that $DG$ is a $\!2\!$-fir.

In this section, we shall obtain such results
for a general monoid $M$ such that $DM$ is a $\!2\!$-fir.
The arguments are unexpectedly complicated, but many
of the intermediate results may be of interest.
We will use left division-closed submonoids in place of
subgroups, so we begin with some observations on left
division closures of commutative submonoids.

\begin{lemma}\label{L.commut}
Let $M$ be a monoid satisfying~\eqref{d.cancel} and~\eqref{d.cap},
and having no nonidentity elements of finite order.
Suppose $a,\,b\in M$ commute, and are
not both $1,$ and let $N$ be the least left division-closed
submonoid of $M$ containing $a$ and $b.$
Then

\textup{(i)} $N=\{c\in M\mid a^i\,b^j\,c\,=\,a^{i'}\,b^{j'}$
for some natural numbers $i,\,j,\,i',\,j'\}.$

\textup{(ii)} $N$ is isomorphic either to the additive monoid $\N$ of
natural numbers, or to the additive group $\Z$ of integers, or to
a submonoid of $\Z\times\Z$
in which the image of $a$ is $(1,0)$ and the image of $b$ is $(0,1).$
\end{lemma}

\begin{proof}
The set given by the right-hand side
of (i) is easily seen to be a submonoid
of $M,$ and to contain $a$ and $b;$
and it is certainly contained in $N,$
since every $c$ as in the description of
that set can be obtained by left-dividing $a^{i'} b^{j'}\in N$
by $a^i b^j\in N.$
Hence to prove~(i), it will suffice to show that
that set is left division-closed.
So suppose $c=d\,e,$ where $c$ and $d$ belong to that set, and $e\in M.$
Writing
\begin{equation}\begin{minipage}[c]{35pc}\notag
\begin{center}
$a^i\,b^j\,c\ =\ a^{i'}\,b^{j'},\qquad a^k\,b^l\,d\ =\ a^{k'}\,b^{l'},$
\end{center}
\end{minipage}\end{equation}
we verify that $e\in N$ by the computation
\begin{equation}\begin{minipage}[c]{35pc}\notag
\begin{center}
$a^{i+k'}\,b^{j+l'} e
\ =\ (a^i b^j)(a^{k'} b^{l'})\,e
\ =\ (a^i b^j)(a^k b^l)\,d\,e
\ =\ (a^i b^j)(a^k b^l)\,c$\\
$=\ (a^k b^l)(a^i b^j)\,c
\ =\ (a^k b^l)(a^{i'} b^{j'})
\ =\ a^{i'+k}\,b^{j'+k}.$
\end{center}
\end{minipage}\end{equation}

To get (ii), let us first note that if an element
$c\in N$ satisfies $a^i b^j c = a^{i'} b^{j'},$ then it also
satisfies $a^{i+m}\,b^{j+n} c = a^{i'+m}\,b^{j'+n}$ for all natural
numbers $m$ and $n.$
This suggests that we associate to each $c\in N$
the set $S_c$ of pairs $(i'\,{-}\,i,j'\,{-}\,j)\in\Z\times\Z$
such that $a^i b^j c = a^{i'} b^{j'}.$
If two such sets are equal, $S_c=S_d,$ it is not hard to see that
we can find $i,$ $j,$ $i',$ $j'$
which simultaneously satisfy $a^i\,b^j\,c = a^{i'}\,b^{j'}$ and
$a^i\,b^j\,d = a^{i'}\,b^{j'};$ so by~\eqref{d.cancel}, $c=d.$
Hence the map $c\mapsto S_c$ is one-to-one.
It is also not hard to verify that for each $c\in N,$ the set
of differences between members of $S_c$
will form a subgroup of $\Z\times\Z,$ and, using~\eqref{d.cancel},
that this subgroup will be the same for all $c\in N.$
Calling this common subgroup $K,$ we find that we have
an embedding $h:N\to(\Z\times\Z)/K.$

I claim that $K$ must be a pure subgroup of $\Z\times\Z,$
i.e., that for any $x\in\Z\times\Z$ and any positive
integer $n,$ if $n\,x\in K$ then $x\in K.$
For given such $x$ and $n,$
we can take elements $c=a^i\,b^j,$ $d=a^{i'}\,b^{j'}$
such that $(i\,{-}\,i',j\,{-}\,j')=x.$
Now the existence of an embedding in $(\Z\times\Z)/K$ shows
that $N$ is commutative, so any two elements have a common
right multiple; so since, by Lemma~\ref{L.N-equiv}\,(i),
$N$ satisfies~\eqref{d.cap}, there exists
$e\in N$ such that either $c\,e=d$ or $d\,e=c.$
Thus, $\pm x\in S_e,$ so the
image of $e$ under the embedding of $N$ in $(\Z\times\Z)/K$
will have exponent $n.$
Since by
hypothesis $M$ has no nonidentity elements of finite order, $e=1,$
so $x\in K,$ showing that $K\subseteq\Z\times\Z$ is indeed pure.

We now consider the rank of $K\subseteq\Z\times\Z$ as an abelian group.

If $K$ has rank zero, i.e., if $K=\{0\},$ then we have an embedding
of $N$ in $\Z\times\Z,$ which by construction
takes $a$ and $b$ to $(1,0)$ and $(0,1)$ respectively,
giving one of the alternative conclusions of~(ii).

If $K$ had rank~$2,$ then
being pure, it would be all of $\Z\times\Z,$
so $N$ would be trivial, contradicting our assumption that
$a$ and $b$ are not both $1.$

This leaves the case where $K$ has rank~$1.$
Since it is pure, we have $(\Z\times\Z)/K\cong \Z,$
so $h:N\to(\Z\times\Z)/K$ yields an embedding $f:N\to\Z.$

If $f(N)\subseteq\Z$ contains both positive and
negative elements, it is clearly a subgroup, hence cyclic,
another of our alternative conclusions.
In the contrary case, we may assume without loss of
generality that it consists of nonnegative integers.
Now condition~\eqref{d.cap}, applied to $f(N)\cong N,$ says that
if $f(N)$ contains both $m$ and $n,$ then their difference, in one
order or the other, also lies in $f(N).$
This allows us to apply the Euclidean algorithm to $f(N),$
and conclude that it is a cyclic monoid, the remaining alternative.
\end{proof}

In the situation
where $DM$ is a $\!2\!$-fir, we can exclude one of the above cases.

\begin{corollary}\label{C.commut}
If $M$ is a monoid such that $DM$ is a $\!2\!$-fir, then
every pair of commuting elements of $M$ lies either in a
cyclic submonoid of $M,$ or in a cyclic subgroup of
the group of invertible elements of $M.$
\end{corollary}

\begin{proof}
As noted in \S\ref{S.conditions},
$M$ satisfies~\eqref{d.cancel} and~\eqref{d.cap}.
Also, $M$ can have no elements of finite order, since
these would lead to zero-divisors in $DM,$ so that it would
not even be a $\!1\!$-fir.
Hence by the preceding lemma, given commuting $a,\,b\in M,$
they will both be contained in a left division-closed
submonoid $N$ of one of the forms described there.
As in the proof of that lemma, we see that
given any two elements of $N,$ one will divide the other.
This implies that $N$ is not only left,
but also right division-closed in $M.$
Since $DM$ is a $\!2\!$-fir,
Lemma~\ref{L.DN<DM}(iii) allows us to conclude that
every $\!2\!$-generator left ideal of $DN$ is projective.

I now claim that if $N$ had the last of the three forms
described in Lemma~\ref{L.commut}(ii),
i.e., if it embedded in $\Z\times\Z$ with $a$ and $b$
mapping to linearly independent elements, then
the left ideal of $DN$ generated by $a-1$ and $b-1$ would be
non-projective.
For if it were projective, then passing to the central localization
$D[a,a^{-1},b,b^{-1}],$ the same would be true of the left ideal
of that ring generated by $a-1$ and $b-1;$
but this ideal is not even flat, as shown by
the method of \cite[Example~4.19]{TYL}.
(The argument is stated there for commutative base ring,
but works equally well in the noncommutative case.)
Alternatively, one can regard $D[a,a^{-1},b,b^{-1}]$ as a group
ring of a finitely generated group, and conclude from the proof of
\cite[p.\,288]{WD+PM} that if its augmentation ideal were
projective, that group would be free, which it is not.
So we must be in one of the other
two cases of Lemma~\ref{L.commut}(ii).
\end{proof}

The next result examines what one can say about
families of more than two commuting elements.
Since the argument used is applicable to elements of finite as well
as infinite order, we do not exclude these until the final sentence.

\begin{proposition}\label{P.cmtgA}
Let $M$ be a monoid satisfying~\eqref{d.cancel}, such that every pair
of commuting elements of $M$ lies either in a cyclic subgroup of the
group of invertible elements of $M,$ or in a cyclic submonoid of~$M.$

Then every finite set $A$ of commuting elements
of $M$ likewise lies either in a cyclic subgroup or
a cyclic submonoid.

Hence, {\em every} set $A$ of commuting elements
of $M$ lies in a submonoid $N$ which is a directed
union either of cyclic subgroups, or of cyclic submonoids.
Such an $N$ is
isomorphic either to a subgroup of the additive group $\Q$ of
rational numbers, or to a subgroup of the factor-group $\Q/\Z$
thereof, or to the monoid of nonnegative elements in a subgroup of $\Q.$

In particular, this is true in every monoid $M$
such that $DM$ is a $\!2\!$-fir.
In that case, since elements of finite order do not occur,
every such $A\not\subseteq\{1\}$ is contained in
a directed union either of infinite cyclic groups,
or of infinite cyclic monoids.
\end{proposition}

\begin{proof}
Let $A$ be a set of commuting elements of $M.$

We first note that if some $a\neq 1$ in $A$ is invertible in $M,$
then every $b\in A$ is invertible.
This is immediate if $b$ lies in a cyclic subgroup
with $a,$ or if $b=1.$
If $b\neq 1$ and $b$ is instead assumed to
lie in a cyclic submonoid containing the invertible element $a,$
then some positive power of $b$ must equal a
power of $a,$ hence must be invertible, so $b$ is invertible.

So either all members of $A$ or no nonidentity
members of $A$ are invertible.
Let us start with the first case.
The hypothesis on pairs of commuting elements of $M$
shows that the subgroup of $M$ generated by any two
commuting invertible elements is cyclic.
Given finitely many commuting invertible elements $a_1,\dots,a_n$
with $n\geq 2,$ assume inductively that the
subgroup generated by $a_1,\dots,a_{n-1}$ is cyclic,
say with generator $a.$
Then $a_n$ will commute with
$a,$ hence will generate with it a cyclic subgroup, which
will be the subgroup generated by $a_1,\dots,a_n,$ as desired.
For possibly infinite $A,$ the above result
shows that each finite subset of $A$ generates a cyclic subgroup, so
the subgroup generated by $A$ is the directed union of those groups.
Such a directed union is easily shown to be isomorphic either to
a subgroup of $\Q$ or to a subgroup of $\Q/\Z.$

In the noninvertible case, we again start with a
finite $A=\{a_1,\dots,a_n\},$ whose members we may assume
without loss of generality are all $\neq 1,$ and we again assume
inductively that $a_1,\dots,a_{n-1}$ are powers
(natural-number powers in this case) of a common element $c.$
Replacing $c$ by a power of itself if necessary, we can
assume that the exponents to which $c$ is raised
to get $a_1,\dots,a_{n-1}$ have no common divisor.
It is easily deduced that the monoid generated by $a_1,\dots,a_{n-1}$
contains both $c^r$ and $c^{r+1}$ for some $r\geq 0.$
Since $a_n$ commutes with each of $a_1,\dots,a_{n-1},$
it commutes with $c^r$ and with $c^{r+1};$ so
$a_n\,c^{r+1}=c^{r+1}\,a_n=c\,a_n\,c^r.$
Cancelling $c^r,$ on the right, we see that $a_n$
commutes with $c,$ hence lies in a cyclic submonoid containing $c,$
and this will thus contain $a_1,\dots,a_n,$ as required.

We see in fact that the above construction yields the
least left division-closed submonoid of $M$ containing $A,$
which is thus also the least cyclic submonoid of $M$ containing $A.$

Looking again at possibly infinite $A,$ we see that $A$ will
be contained in the directed union of the least cyclic submonoids
containing its finite subsets.
Since the elements of $A$ are not invertible, and $M$ has
cancellation, those cyclic submonoids are infinite; and
it is not hard to show that
a directed union of infinite cyclic monoids has the form claimed.

Corollary~\ref{C.commut} shows that the above results apply
to monoids $M$ such that $DM$ is a $\!2\!$-fir.
As we have noted, even for $DM$ a $\!1\!$-fir, $M$ cannot
have elements of finite order, giving the final assertion.
\end{proof}

We now wish to prove a corresponding result for
the {\em centralizer} of any nonidentity element $c\in M.$
From Corollary~\ref{C.commut}, we see that every nonidentity element
of the centralizer of $c$ has a positive or negative power which is a
positive power of $c;$ hence, any two elements
$a$ and $b$ of that centralizer have nonzero powers which are equal.
For elements of a general group or monoid, this does not
guarantee that $a$ and $b$ commute.
For instance, in the group
\begin{equation}\begin{minipage}[c]{35pc}\label{d.x^2=y^2}
$\lang x,\,y\mid x^2=y^2\rang\gp,$
\end{minipage}\end{equation}
$x$ and $y$ each lie in the centralizer of $x^2=y^2\neq 1,$
and each generates, with that element, a cyclic subgroup,
but $x$ and $y$ do not commute with each other.
But we shall see that this does not happen
in the monoids we are interested in.
The proofs of this fact will be quite different
depending on whether $a$ and $b$ are invertible.
We start with the noninvertible case.

(Ferran Ced\'{o} informs me that the next two results are related to
the cited results from the unpublished thesis~\cite{AP}.)

\begin{lemma}[cf.~{\cite[Corol$\cdot$lari~3.2.2]{AP}}]\label{L.>cm}
Let $M$ be a monoid satisfying~\eqref{d.cancel}-\eqref{d.a=cad>},
and suppose $a,\,b$ are noninvertible elements of $M$
such that $a^m=b^n$ for some positive integers $m$ and $n.$
Then $a$ and $b$ commute.
\end{lemma}

\begin{proof}
The products $a b$ and $b a$ have a common
right multiple: $a\,b^{2n} = a^{2m+1} = b^{2n} a = b\,b^n\,b^{n-1}\,a
= b\,a^m\,b^{n-1}\,a.$
Hence by~\eqref{d.cap}, one of them is a right multiple of the other.
Without loss of generality say
\begin{equation}\begin{minipage}[c]{35pc}\label{d.ab=bac}
$a\,b\ =\ b\,a\,c.$
\end{minipage}\end{equation}
It follows that $a^m b = b (ac)^m.$
Since $a^m=b^n$ commutes with $b,$ we can cancel the $b$ in that
equation on the left, and get $a^m=(ac)^m.$
Since $a$ is noninvertible, Lemma~\ref{L.a1a2>} tells us that $c=1,$
so~\eqref{d.ab=bac} says that $a$ and $b$ commute.
\end{proof}

The analog of the above lemma fails if $a$ and $b$ are invertible, for
every group satisfies~\eqref{d.cancel}-\eqref{d.ab=ca>},
but groups such as~\eqref{d.x^2=y^2} have
noncommuting elements with commuting powers.
But if $DM$ is a $\!2\!$-fir, that behavior can
be excluded using Theorem~\ref{T.n-gen_sgp}(ii).
Let us combine that argument with the preceding results of this section,
and prove

\begin{theorem}[cf.~{\cite[Teoremas~3.2.2,~3.2.9]{AP}}]\label{T.cm}
If $M$ is a monoid such that $DM$ is a $\!2\!$-fir, then the
binary relation of mutual commutativity is an equivalence relation
on the nonidentity elements of $M,$
and each equivalence class either \textup{(i)}~consists
entirely of invertible elements, or \textup{(ii)}~consists
entirely of noninvertible elements.

Every equivalence class falling under~\textup{(i)} forms,
together with the identity element, a locally infinite-cyclic
subgroup of the invertible elements of $M;$ equivalently,
a group embeddable in the additive group of $\Q.$
Every equivalence class falling under~\textup{(ii)} forms,
together with the identity element, a locally infinite-cyclic
submonoid of $M;$ equivalently, a submonoid isomorphic
to the monoid of nonnegative elements in a subgroup of $\Q.$
\end{theorem}

\begin{proof}
We know from the proof of Proposition~\ref{P.cmtgA}
that of two commuting nonidentity elements of $M,$ either both
or neither are invertible.
To show that commutativity is an equivalence relation on such
elements, suppose $a$ commutes with $b,$ and $b$ with $c.$
By Proposition~\ref{P.cmtgA}, $a$ and $b$ will lie in
a common cyclic subgroup (if $b$ is invertible)
or submonoid (otherwise), and likewise for $b$ and $c.$
Hence each of $a$ and $c$ has a nonzero power (positive in
the case where $b$ is noninvertible) which equals a power of $b.$
Hence raising these powers to further powers, we have
$a^m=c^n$ for some nonzero $m$ and $n.$
If $b$ is noninvertible, Lemma~\ref{L.>cm}
now tells us that $a$ and $c$ commute.
In the invertible case, Theorem~\ref{T.n-gen_sgp}(ii) shows that
the subgroup of $M$ generated by $a$ and $c$ is free, so, as it has
a nontrivial central
element $a^m=c^n,$ it must be cyclic, so $a$ and $c$ again commute.

Applying Proposition~\ref{P.cmtgA} to any equivalence class $A$
of noninvertible elements
under the commutativity relation, we see that it will be
contained in a commutative subgroup or submonoid $N$ of the
asserted form; and since $N$ is itself commutative, we must
have $N=A\cup\{1\}.$
\end{proof}

We remark that for $M$ a monoid such that $DM$ is a {\em semifir}, the
fact that commutativity is an equivalence relation on $M$
follows from conditions~\eqref{d.G_l_f} and~\eqref{d.M<G}, while for
$M$ a {\em group} such that $DM$ a $\!2\!$-fir,
the same fact follows from Dicks and Menal's result that
every $\!2\!$-generator subgroup of $M$ is free.
But when $M$ is a monoid, the above result gives
a new necessary condition for $DM$ to be a $\!2\!$-fir.

\section{Miscellaneous observations}\label{S.C_etc}

\subsection{Two more questions}\label{SS.D}
We have mentioned that necessary and sufficient conditions are
known for a monoid ring $DM$ to be a right or left fir, and, when $M$
is a group, for $DM$ to be a semifir.
In each of these conditions, if $M$ is nontrivial then $D$ must be
a division ring, but aside from that, the conditions
are restrictions on $M$ alone.
It seems likely that for a general monoid $M$
the same will be true of the condition
for $DM$ to be a semifir, or more generally, an $\!n\!$-fir.
But we don't know this, so we pose it as

\begin{question}\label{Q.same_D}
Let $n$ be a positive integer.
If $M$ is a monoid such that $DM$ is an $\!n\!$-fir
for some division ring $D,$ will
$D'M$ be an $\!n\!$-fir for every division ring $D'$?
\end{question}

Monoid rings can be generalized to ``skew'' monoid rings; so we ask

\begin{question}\label{Q.twist}
Let $M$ be a monoid, $D$ a division ring, and $\alpha:M\to\r{Aut}(D)$
a homomorphism, with the image under $\alpha$ of $x\in M$
denoted $\alpha_x.$
Let $D[M;\alpha]$ denote the ``skew'' monoid ring
consisting of all left-linear expressions
$\sum_{x\in M} c_x x$ $(c_x\in D,$ zero for almost all $x),$ with
multiplication and addition defined among elements of $D$ as in $D,$
and multiplication defined among elements of $M$ as in $M,$
while in place of the mutual commutativity of $D$ and $M$ assumed
in ordinary monoid rings $DM,$ one uses the relations
\begin{equation}\begin{minipage}[c]{35pc}\label{d.twist}
$x\,c\ =\ \alpha_x(c)\,x\quad (x\in M,\,c\in D).$
\end{minipage}\end{equation}

Is it true that $D[M;\alpha]$ will be a left fir, a right
fir, a semifir, etc., if and only if the ordinary
monoid ring $DM$ has the same property?
\end{question}

There are still other variants of the monoid ring construction.
For instance, we can get new multiplications~$*$
on $DM$ by defining, for each $x,\,y\in M,$
$x*y= c_{x,y}\,x y,$ where $(c_{x,y})_{x,y\in M}$
are nonzero elements of $D$ satisfying the
relations required to make this multiplication associative.
Rings constructed using such $c_{x,y},$ and possibly
also a map $\alpha:M\to\r{Aut}(D)$ as above, are called
``crossed products'' of $D$ and $M;$ and
one can ask when these are $\!n\!$-firs.

\subsection{Some properties of our properties}\label{SS.prop_prop}

If $R$ is a ring and $n$ a positive integer,
then the condition that $R$ be an $\!n\!$-fir is what logicians
call a {\em first order} property.
This means, in our context, that it is equivalent to
the conjunction of a (possibly
infinite) set of conditions, each of which is expressible using
ring-theoretic equations, logical connectives, and quantifications
over elements (but not quantifications over subsets).
It is not hard to see that the condition that every $\!m\!$-term
linear relation be trivializable may be written in this form.
(The formulation of the $\!n\!$-fir condition that we began
with, saying that every left ideal of $R$ generated by $\leq n$
elements is free, and
free modules of ranks $\leq n$ have unique ranks, takes a
little more thought, but is also not hard to express in such a form.)
Gathering together, for all $n,$ these first-order statements defining
$\!n\!$-firs, we see that
the property of being a semifir is also first-order.

Hence it is curious that the condition on a monoid $M,$ that all
monoid rings $DM$ be $\!n\!$-firs or semifirs (or, to avoid
worrying about whether this condition depends on $D,$
that the monoid ring $DM$ be an $\!n\!$-fir or semifir for
a fixed $D,$ e.g., the field of rational numbers), is {\em not}
first-order.
For instance, the additive group of $\Z,$ and a
nonprincipal ultraproduct of countably many copies
of that group, are {\em elementarily equivalent},
that is, they satisfy the same first-order conditions;
but such an ultraproduct group will contain two linearly independent
elements, and hence fail to satisfy~\eqref{d.G_l_f}.

Here is another observation of the same sort,
though not as close to the main topic of this note:
Though the class of semifirs is first-order, we shall see
that the class of rings in which all finitely generated
left ideals are free, but not necessarily of unique rank, is not.

This looks like a justification for considering the class of semifirs
``nicer'' than the wider class.
While I personally prefer the study
of semifirs, the above observation is not
a justification for that preference; it simply reflects the fact
that the property that all finitely generated left ideals be free
comprises too complicated a family of cases to be first order.
If one specifies for {\em which} pairs of positive
integers $m$ and $n$ one wants to have $R^m\cong R^n$
(which corresponds to specifying a congruence on the additive semigroup
of positive integers), and looks at
the class of rings $R$ over which
those and only those isomorphisms hold,
and all finitely generated left ideals are free,
the resulting class is first order.
The non-first-order class referred to above
is the union of this countably infinite family of first-order classes.
Each of these classes is, incidentally, nonempty
\cite{PMC.depII}, \cite[Theorem~6.1]{cPu}.

To see that, as claimed, this union is not first order, suppose we take,
for each $d>1,$ a ring $R_d$ whose finitely
generated left ideals are free, and such that $R_d^m\cong R_d^n$
for positive integers $m$ and $n$ if and only if $d\,|\,m-n.$
Then for each $d,$ the free $\!R_d\!$-module of rank $1$
has a decomposition $R_d\cong R_d\oplus R_d^d.$
It is easy to deduce that over
a nonprincipal ultraproduct $R$ of the $R_d,$ we get an
$\!R\!$-module decomposition $R\cong R\oplus P,$ where $P$ is a
direct summand in $R$ which has
free modules of all positive ranks as direct summands.
But $P$ cannot itself be free, since if it were, then
being finitely generated, it would have to be free
of some finite rank $n,$ implying an isomorphism
$R\cong R^{n+1};$ but among the $R_d,$ this isomorphism only holds for
finitely many of those rings, namely those for which $d\,|\,n,$
hence it does not hold in the ultraproduct $R.$

Let us also note that though the class of semifirs is
closed under coproducts with amalgamation of a common division
ring \cite{PMC.cP}, \cite[p.\,201]{HB}, \cite[Corollary~2.13]{cP},
the class of groups $G$ such that $DG$ is a semifir
is not closed under coproducts with amalgamation of a common subgroup.
A counterexample is the group~\eqref{d.x^2=y^2} above, a
coproduct of two infinite cyclic
groups with amalgamation of a subgroup of index two in each.

\subsection{On condition~\eqref{d.ab=ca>}, and possible variants}\label{SS.var_ab=ca}
Condition~\eqref{d.ab=ca>} on a monoid $M$ says that
relations $a\,b=c\,a$ (with $b$ and $c$ not both $1)$ in $M$
can only arise in a ``generic'' way, namely, by taking
elements $e,\,f\in M,$ and a natural number $n,$ and
letting $a=(ef)^n e,$ $b=fe,$ and $c=ef.$

Yet one can clearly also get such relations, more generally,
by taking elements $e$ and $f$ and natural numbers $m$ and $n$ with
$m>0,$ and letting
\begin{equation}\begin{minipage}[c]{35pc}\label{d.e,f,m,n}
$a=(ef)^n e,$\quad $b=(fe)^m,$\quad and\quad $c=(ef)^m.$
\end{minipage}\end{equation}

The solution to this paradox is that~\eqref{d.e,f,m,n}
can be rewritten in the form described in~\eqref{d.ab=ca>}.
Namely, writing $n=qm+r$ with $0\leq r< m,$
and setting $\bar{e}=(ef)^r e$ and
$\bar{f}=f (ef)^{m-r-1},$ we find that
$a=(\bar{e}\bar{f})^q\,\bar{e},$
$b=\bar{f}\bar{e},$ and $c=\bar{e}\bar{f},$ as desired.

Now suppose that for some $a\in M$ we have a {\em pair} of relations,
\begin{equation}\begin{minipage}[c]{35pc}\label{d.ab=ca&}
$a\,b\ =\ c\,a,$\qquad $a\,b'\ =\ c'\,a,$
\end{minipage}\end{equation}
with $b,$ $c,$ $b',$ $c'\in M.$
I see several ``generic'' patterns for how families
of elements satisfying these equations can arise.
In writing the first of these, I will, for
visual simplicity, use $[s]^{-1}$ to denote {\em deletion} of
a string $s$ from the beginning or end of an expression.
The first pattern is then
\begin{equation}\begin{minipage}[c]{35pc}\label{d.efg}
$a=((ef)^m e g)^n (ef)^m e,\quad
b=g (ef)^m e,\quad
c=(ef)^m e g,\quad
b'=[efe]^{-1} a,\quad
c'=a[efe]^{-1},$\\
where $m\geq 1,$ $n\geq 0.$
\end{minipage}\end{equation}

This pattern has a dual, in which the formulas given above for
$b$ and $c$ are instead used for $b'$ and $c',$ and vice versa.
The reader can easily write this down.

The next pattern (which has some intersection with what
we get on taking $g=f$ in~\eqref{d.efg}), is
\begin{equation}\begin{minipage}[c]{35pc}\label{d.all_ef}
$a=(ef)^n\,e,$\quad
$b=(fe)^p,$\quad
$c=(ef)^p,$\quad
$b'=(fe)^q,$\quad
$c'=(ef)^q,$\\
where $p$ and $q$ are relatively prime,
and $0\leq p\leq n$ and $0\leq q\leq n.$
\end{minipage}\end{equation}
(We assume $p$ and $q$ relatively prime because the case where
they have a common divisor can be reduced to~\eqref{d.all_ef}
by a substitution like the one we described
following~\eqref{d.e,f,m,n} above.
We have also excluded the case where $p>n,$ because in such
cases, $a$ right-divides $b,$ and these fall under~\eqref{d.disj} below;
and the case where $q>n$ for the corresponding reason.)

Finally, we have the following pattern,
and the variant we get by interchanging the
roles of $b,\ c$ with those of $b',\ c'.$
\begin{equation}\begin{minipage}[c]{35pc}\label{d.disj}
$a=(ef)^n\,e,$\quad
$b=g\,a,$\quad
$c=a\,g,$\quad
$b'=fe,$\quad
$c'=ef.$
\end{minipage}\end{equation}
(One can also write down a pattern in which
both $b$ and $b'$ are left multiples of $a.$
But this reduces to the $n=0$ case of~\eqref{d.disj}.)

\begin{question}\label{Q.patterns}
\textup{(i)}
If $M$ is a monoid such that $DM$ is a $\!2\!$-fir,
is it true that for all $a,$ $b,$ $b',$ $c,$ $c'\in M$
satisfying~\eqref{d.ab=ca&} with $b,$ $b',$ $c,$ $c'\neq 1,$
either~\eqref{d.efg}, or~\eqref{d.all_ef},
or~\eqref{d.disj}, or the dual of~\eqref{d.efg}
or~\eqref{d.disj}, must hold for some $e,$ $f,$ etc.\ in $M$
and natural numbers $m,$ $n,$ etc.?

\textup{(ii)} If the answer to \textup{(i)} is no, does there
exist a finite list of decompositions which {\em will} always
yield the given relations?

\textup{(iii)} If \textup{(i)} or \textup{(ii)} has a positive
answer, i.e., if the existence of decompositions of the
relevant sort whenever~\eqref{d.ab=ca&} holds is a necessary
condition for $DM$ to be a $\!2\!$-fir, is this condition
implied by~\eqref{d.cancel}-\eqref{d.ab=ca>}
\textup{(}as, for instance,~\eqref{d.a1a2>} is\textup{)},
or does adding it to that list give a
stronger set of conditions on a monoid~$M$?

\textup{(iv)} Does the answer to any of the above questions
change if one strengthens the condition that $DM$ be a
$\!2\!$-fir to say that it is an $\!n\!$-fir, for some $n>2$?
\end{question}

To study the consequences of~\eqref{d.ab=ca&} in a
$\!2\!$-fir, one might use the $\!2\!$-term
linear relation analogous to~\eqref{d.a(b-1)},
\begin{equation}\begin{minipage}[c]{35pc}\label{d.a(b+b'-1)}
$a\,(b+b'-1)-(c+c'-1)\,a\ =\ 0.$
\end{minipage}\end{equation}

Instead of fixing an element $a,$ and using
more than one pair $b,$ $c,$ as in~\eqref{d.ab=ca&},
one might do the opposite, and study pairs of relations,
\begin{equation}\begin{minipage}[c]{35pc}\label{d.ab=ca&&}
$a\,b=c\,a,$\qquad $a'\,b=c\,a',$
\end{minipage}\end{equation}
in the same spirit.

Another family of ``generic'' relations, this time ring-theoretic rather
than monoid-theoretic, is the list of $\!2\!$-term linear relations
holding in any ring containing ring-elements $x,\,y,\,z,\,\dots$:
\begin{equation}\begin{minipage}[c]{35pc}\label{d.leapfrog}
\begin{center}
$1\cdot x=x\cdot 1,$\\
$x\cdot (yx+1)=(xy+1)\cdot x,$\\
$(xy+1)\cdot(zyx+z+x)=(xyz+x+z)\cdot(yx+1),$\\
$(xyz+x+z)\cdot(wzyx+wz+wx+yx+1)=(xyzw+xy+xw+zw+1)\cdot(zyx+z+x),$\\
$.\ \ .\ \ .$
\end{center}
\vspace{.2em}
\end{minipage}\end{equation}
These are defined and studied in \cite[\S2.7]{FRR+},
and it is shown in \cite[Proposition~2.7.3]{FRR+} that in an important
class of $\!2\!$-firs, they give all $\!2\!$-term
linear relations whose pair of left factors has no common
left divisor and whose pair of right factors has no common
right divisor, equivalently, which can be
trivialized to $1\cdot 0=0\cdot 1.$

By examining the monoid relations holding among the summands
in these identities,
we can write down families of relations in a monoid, which might
(who knows?)\ be particularly useful in finding conditions for
monoid rings $DM$ to be $\!2\!$-firs.
For example, looking at the second of the above identities,
$x\cdot (yx+1)=(xy+1)\cdot x,$ we see that if we
give $x,$ $yx$ and $xy$ the names $a,$ $b$ and $c$ respectively,
then that ring-theoretic equation follows from the monoid
relation $a\,b=c\,a$ that these elements satisfy, which is the subject
of~\eqref{d.ab=ca>}.
If we examine the monoid-theoretic relations corresponding
in the same way to the next identity of~\eqref{d.leapfrog}, we get,
with more work, the same single equation.
But moving to the fourth line of~\eqref{d.leapfrog},
we find, among the monomials in the factors comprising that
linear relation, eight that
are not products of any others, namely
$x,$ $z,$ $wz,$ $wx,$ $yx,$ $xy,$ $xw$ and $zw;$
and calling these $a,$ $b,$ $c,$ $d,$ $e,$ $f,$ $g$ and $h,$
we find that the fifteen monoid relations among
these elements that underlie that
ring identity reduce, after tautologies and
repetitions are dropped, to five:
\begin{equation}\begin{minipage}[c]{35pc}\label{d.five}
$ac=gb,$\quad $ad=ga,$\quad $ae=fa,$\quad $bc=hb,$\quad $bd=ha.$
\end{minipage}\end{equation}
So it might be of interest to search for results to the effect
that if this family of relations holds in a monoid $M$
such that $DM$ is a $\!2\!$-fir, then it can be achieved in
some ``generic'' way; and, if this approach turns out to be productive,
the identities further down on the list beginning
with~\eqref{d.leapfrog} might be investigated similarly.

What about conditions for monoid rings to be $\!n\!$-firs for
higher values of $n$?
For $n>2,$ the process of trivializing an $\!n\!$-term relation
seems to have a much more fluid character than for $n=2.$
E.g., compare the simplicity of the Euclidean algorithm for
pairs of real numbers with the complexity, and lack of a single
natural choice, among the algorithms for larger families
of real numbers discussed in~\cite{F+F}.
Nonetheless, one may hope for further
results on the conditions for the $\!n\!$-fir property to hold.

\subsection{An observation on $\sigma^{1/2}$}\label{SS.sigma^1/2}
I claim that the fact that the automorphism $\sigma^{1/2}$
of the free group on $x$ and $y$ introduced
in Lemma~\ref{L.sigma^1/2} sends
$z=x^{-1}y^{-1}xy$ to $z^{-1}$ is (loosely) related to the fact that
$\tau=(1+\sqrt 5)/2$ has algebraic norm $-1$ over $\Q.$
Precisely, let me sketch an argument using the latter
fact which shows that in $G/[G,G'],$ the image of $\sigma^{1/2}(z)$
must be the inverse of the image of $z.$

It is not hard to show that
the commutator operation on any group $G$ induces an
alternating bilinear map $G/G'\times G/G'\to G'/[G,G'],$
whose image generates the latter group.
Hence for $G$ the free group of rank $2,$ which has
$G/G'\cong\Z\times\Z$ and $G'/[G,G']\cong\Z,$ any endomorphism $\alpha$
of $G$ must induce on $G'/[G,G']$ the operation of exponentiation by
the determinant of the map that $\alpha$ induces on $G/G'.$
Now we have seen that the automorphism of $G/G'$ induced
by $\sigma^{1/2}$ acts
as does multiplication by $\tau$ on $\Z+\tau\Z\subseteq\R;$
and the determinant of that automorphism of free
abelian groups is the norm of $\tau,$ namely $-1.$
So since $z\in G',$ the image of $z$ in $G'/[G,G']$ must
indeed be sent to its inverse.

For some observations on the limiting behavior of the orbit
of $x$ under $\sigma^{1/2}$ under group topologies
on $G,$ see~\cite[\S7]{group_top}.

\section{Acknowledgements}\label{S.ackn}
I am grateful to Warren Dicks and Ferran Ced\'{o} for
some very helpful correspondence about this material, and
to the referee for some suggestions and corrections which are
followed in the version to appear in {\em Semigroup Forum}.

\section{Typographical note}\label{S.lang_rang}
The angle brackets used
in this note, $\lang\ \rang,$ are not the standard {\TeX}
symbols $\langle\ \rangle.$
For the {\TeX} code for the symbols used here, and
my thoughts on the subject, see~\cite{lang_rang}.
(However, in this note, since the symbols enclosed in
angle brackets are almost all lower-case, I have added
\texttt{\symbol{92}kern.16em} to the definition of
\texttt{\symbol{92}lang}
shown in \cite{lang_rang}, so as to get, for example,
$\lang x\rang$ rather than $\lang\kern-.16em x\rang.)$

\end{document}